\documentclass{amsart}

\usepackage[T1]{fontenc}
\usepackage{mathptmx}
\usepackage[utf8]{inputenc}
\usepackage{amsmath,amsthm,amsfonts,amssymb}
\usepackage{xcolor}
\usepackage{dsfont}
\usepackage{mathtools}
\usepackage[all,cmtip]{xy}
\usepackage{tikz}
\usepackage{tikz-cd}
\usepackage{thmtools}
\usepackage{tracefnt}

\usepackage{hyperref}
\hypersetup{
	colorlinks=true,
	linktocpage=true,
	linkcolor=[RGB]{161,1,49},
	citecolor=[RGB]{10,136,169},      
	urlcolor=cyan,
}

\allowdisplaybreaks

\newtheorem{theorem}{Theorem}
\numberwithin{theorem}{section}

\newtheorem{corollary}[theorem]{Corollary}
\newtheorem{lemma}[theorem]{Lemma}
\newtheorem{proposition}[theorem]{Proposition}

\theoremstyle{remark}
\newtheorem{remark}{Remark}
\theoremstyle{definition}
\newtheorem{example}{Example}

\usepackage{comment}
\usepackage[all,cmtip]{xy}

\newcommand{\R}{\mathbb{R}}
\newcommand{\Z}{\mathbb{Z}}
\newcommand{\N}{\mathbb{N}}
\newcommand{\acts}{\curvearrowright}
\newcommand{\Isom}{\mathrm{Isom}}
\newcommand{\FSet}{\mathcal{F}(G)}
\newcommand{\Part}{\mathcal{P}}
\newcommand{\Maps}{\mathcal{M}}

\newcommand{\NMaps}{\mathcal{N}}

\newcommand{\AdMaps}{\Maps_\oplus}
\newcommand{\AdNMaps}{\NMaps_\oplus}
\newcommand{\AsMaps}{\overline{\AdMaps}}
\newcommand{\Cob}{\mathcal{L}}
\newcommand{\wCob}{{\overline{\Cob}}}

\newcommand{\witness}{\mathcal{R}_\varphi}

\newcommand{\vertiii}[1]{{\left\vert\kern-0.25ex\left\vert\kern-0.25ex\left\vert #1 
    \right\vert\kern-0.25ex\right\vert\kern-0.25ex\right\vert}}

\newcommand{\vertG}[1]{{\left\vert\kern-0.25ex\left\vert\kern-0.25ex\left\vert #1 
    \right\vert\kern-0.25ex\right\vert\kern-0.25ex\right\vert}_G}

\newcommand{\vertsup}[1]{{\left\vert\kern-0.25ex\left\vert\kern-0.25ex\left\vert #1 
    \right\vert\kern-0.25ex\right\vert\kern-0.25ex\right\vert}_\infty}

\newcommand{\vertK}[1]{{\left\vert\kern-0.25ex\left\vert\kern-0.25ex\left\vert #1 
    \right\vert\kern-0.25ex\right\vert\kern-0.25ex\right\vert}_\NMaps}

\title[Additive realizations of asymptotically additive set maps]{Additive realizations of asymptotically additive set maps}
\author{Raimundo Briceño}
\author{Godofredo Iommi}
\address{Facultad de Matem\'aticas, Pontificia Universidad Cat\'olica de Chile. Santiago, Chile}
\email{raimundo.briceno@uc.cl}
\urladdr{http://www.mat.uc.cl/~raimundo.briceno}
\email{godofredo.iommi@gmail.com}
\urladdr{http://www.mat.uc.cl/~giommi}

\date{}

\begin{document}

\begin{abstract}
Given a countable discrete amenable group, we study conditions under which a set map into a Banach space (or more generally, a complete semi-normed space) can be realized as the ergodic sum of a vector under a group representation, such that the realization is asymptotically indistinguishable from the original map. We show that for uniformly bounded group representations, this property is characterized by the class of bounded asymptotically additive set maps, extending previous work for sequences in Banach spaces and on the case of a single non-expansive linear map. Additionally, we develop a relative version of this characterization, identifying when the additive realization can be chosen within a prescribed target set. As an application, our results generalize central aspects of thermodynamic formalism, bridging the additive and asymptotically additive frameworks.
\end{abstract}

\maketitle
\setcounter{tocdepth}{2}
\tableofcontents

\section*{Introduction}

Non-additive ergodic theory was originally introduced to study the maximal Lyapunov exponent of matrix-valued functions \cite{furstenberg1960products} and subsequently motivated by questions in percolation theory, statistical mechanics, thermodynamic formalism, dimension theory, among other areas \cite{akcoglu1981,barreira3, derriennic1975, dooley2014sub,falconer1988, hammersley1965first,krengel2011ergodic,nguyen1979ergodic2,smythe1976multiparameter}. A prominent example is the sub-additive ergodic theorem proved by Kingman in 1968 \cite[Theorem 5]{kingman1968}. Given a measure-preserving transformation $T$ on a probability space $(X,\mu)$, a sequence of real-valued functions $(f_n)_n$ in $L^1_\mu(X)$ is said to be \emph{sub-additive} if for every $n,m \in \N$ and $\mu$-almost every $x \in X$, we have $f_{n+m}(x) \leq f_n(x) + f_m(T^n x)$. Kingman's result states that if $f_1^+ \in L^1_\mu(X)$, then there exists $\overline{f} \in L^1_\mu(X)$ such that for $\mu$-almost every $x \in X$, we have
\begin{equation*}
\lim_{n \to \infty} \frac{f_n(x)}{n}  = \overline{f}(x) \quad \text{and} \quad
\lim_{n \to \infty} \frac{1}{n} \int{f_n(x)}d\mu= \int{\overline{f}(x)}d\mu.
\end{equation*}
The original proof is based on the fact that there exist $f \in L^1_\mu(X)$ and a sub-additive sequence of positive functions $(r_n)_n$ such that for $\mu$-almost every $x \in X$, we have
\begin{equation} \label{king_dec}
f_n(x) =  \sum_{i=0}^{n-1} f(T^ix) + r_n(x)
\quad \text{ and } \quad \lim_{n \to \infty} \frac{r_n(x)}{n}=0.
\end{equation}
That is, the sub-additive sequence $(f_n)_n$ is the sum of an additive sequence and a sub-linear error $\mu$-almost everywhere (see \cite[\S 5]{kingman1968}). Once a decomposition like this is obtained, the sub-additive ergodic theorem can be deduced from Birkhoff's theorem by means of the function $f$.

Of course, the decomposition in Equation \eqref{king_dec} holds in the \emph{measure-theoretic} category, since it depends on $\mu$. It turns out that, under suitable additive assumptions on the sequence $(f_n)_n$, one can obtain a decomposition analogous to that in Equation \eqref{king_dec} in the \emph{topological} category. Specifically, given a continuous transformation $T$ on a compact metric space $(X,d)$, a sequence of real-valued functions $(f_n)_n$ in $\mathcal{C}(X)$ is said to be \emph{asymptotically additive} if
\begin{equation} \label{aa-fh}
\inf_{f \in \mathcal{C}(X)} \limsup_{n \to \infty} \frac{1}{n} \left\|  f_n - \sum_{i=0}^{n-1} f \circ T^i  \right\|_{\infty} =0,
\end{equation}
where $\mathcal{C}(X)$ is the space of continuous functions and $\|\cdot\|_{\infty}$ the uniform norm. This notion was introduced by Feng and Huang \cite{feng-huang} in their study of Lyapunov spectra of certain functions. Later, Cuneo \cite[Theorem 1.2]{cuneo2020additive} proved that if $(f_n)_n$ is an asymptotically additive sequence, then there exist $f \in \mathcal{C}(X)$ and a sequence $(r_n)_n$ in $\mathcal{C}(X)$ such that, for every $x \in X$, we have
\begin{equation} \label{dec_cun}
f_n(x)= \sum_{i=0}^{n-1} f(T^ix) + r_n(x) \quad \text{and} \quad 
\lim_{n \to \infty} \frac{\|r_n\|_{\infty}}{n}=0.
\end{equation}
In other words, an asymptotically additive sequence is additive up to sublinear corrections. More generally, Cuneo proved \cite[Theorem 2.1]{cuneo2020additive} that if $(V,\|\cdot\|)$ is a Banach space and $K: V \to V$ is a non-expansive linear operator---meaning that its operator norm satisfies $\|K\|_\mathrm{op} \leq 1$---, then for any asymptotically additive sequence $(v_n)_n$ in $V$, that is,
$$
\inf_{v \in V} \limsup_{n \to \infty} \frac{1}{n}\left\|v_n-\sum_{i=0}^{n-1} K^i v\right\|=0,
$$
there exists $v \in V$ such that
$$
\lim_{n \to \infty} \frac{1}{n}\left\|v_n-\sum_{i=0}^{n-1} K^i v\right\|=0.
$$
In particular, this applies to the Koopman operator associated to $T$. As in Kingman's theorem, such a conclusion permits the reduction of several questions involving non-additive sequences to the classical additive setting.

Decomposing a non-additive phenomenon into an additive one plus a negligible error, also plays an important role in the study of ergodic properties of spatial stochastic processes. This setting usually takes place in a space of the form $X = \Omega^G$, where $\Omega$ is a topological space---such as a finite set with the discrete topology---and $G$ is a group---such as $\Z^d$. Then, for a Borel probability measure $\mu$, it is natural to consider \emph{set maps} $\varphi: \FSet \to L^1_\mu(X)$, where $\FSet$ denotes the collection of non-empty finite subsets of $G$, rather than sequences of functions $(f_n)_n$ in $L^1_\mu(X)$. Here, a set map $\varphi$ may represent a quantity such as information, volume, energy, etc., so that $\varphi(F)(x)$ denotes the amount of such quantity in the region $F$ for $x \in X$. Generalizations of Kingman's theorem and analogues of the decomposition in Equation \eqref{king_dec} have been obtained in this context; for example, see \cite[Proposition 2.1]{smythe1976multiparameter} and \cite[\S  4.19]{nguyen1979ergodic2} for the $G = \Z^d$ case. Note however that the aforementioned decomposition does not hold in general \cite[\S 5.1]{akcoglu1981} and extra assumptions, like \emph{strong sub-additivity}, are required. Although the notion of asymptotically additivity has appeared in the context of multidimensional sequences of real-valued continuous functions and $\Z^d$-actions (see \cite[Definition 5.2]{yan2013sub}), no analogue of Cuneo's theorem has been obtained.

In this article we generalize Cuneo's result in several directions. Instead of considering a single non-expansive linear map $K: V \to V$ from a Banach space $V$ to itself---which can be thought of, in the invertible case, as a representation of the group of integers $\Z$---, we consider uniformly bounded group representations $\pi: G \to \Isom(V)$ from an arbitrary countable discrete amenable group $G$ into the space $\Isom(V)$ of invertible linear operators on a complete semi-normed space $V$, the Koopman representation being the most natural example. Also, instead of sequences of functions $(f_n)_n$, we consider set maps $\varphi: \FSet \to V$. Thus, our results consider more general groups, representations, set maps, and spaces.

We believe that the correct setting to study these questions in the generality we sought is the complete semi-normed space we introduce in \S \ref{sec:space-maps}. For a countable discrete amenable group $G$, a complete semi-normed (or in particular, Banach) space $(V,\|\cdot\|)$, and a uniformly bounded representation $\pi: G \to \mathrm{Isom}(V)$, we define the complete semi-normed space  $(\Maps,\vertsup{\cdot})$ of \emph{bounded} and \emph{$G$-equivariant} set maps $\varphi$ (see \S \ref{sec:space-maps} for the precise definitions). In this context, the class of additive maps $\AdMaps$ forms a closed subspace of $\Maps$. Of course, in  such a general setting appropriate additive notions are required.  We say that a set map $\varphi \in \Maps$ is \emph{asymptotically additive} if 
\begin{equation} \label{defi_aa}
\inf_{\psi \in \AdMaps} \limsup_{F \to G}\left\|\frac{\varphi(F)}{|F|} -\frac{\psi(F)}{|F|}\right\| =0,
\end{equation}
where the limit superior is along arbitrary F{\o}lner sequences (see \S \ref{sec:limits} for definitions). It turns out that this notion is more general and encompassing than the one introduced by Feng and Huang \cite{feng-huang} (see Equation \eqref{aa-fh}) and used by Cuneo even for $\Z$-actions. Considering this, our first result is the following.

\begin{restatable}{thm}{theoremone}
\label{thm:cuneo}
Let $G$ be a countable amenable group, $(V,\|\cdot\|)$ a complete semi-normed space $\pi: G \to \mathrm{Isom}(V)$ a uniformly bounded representation. If a set map $\varphi$ is asymptotically additive, then there exists an additive set map $\psi$ such that
$$\lim_{F \to G}\left\|\frac{\varphi(F)}{|F|} -\frac{\psi(F)}{|F|}\right\| = 0.$$
\end{restatable}
That is, we establish that infimum in Equation \eqref{defi_aa} is actually attained, generalizing in this way Cuneo's original result for single non-expansive linear maps of Banach spaces. If we abbreviate
$$
A_F v= \frac{1}{|F|}\sum_{g \in F} \pi(g^{-1})v \quad   \text{for}~v \in V,
$$
we have the following consequence.

\begin{restatable}{cor}{corollaryone}
\label{cor:cuneo}
Let $G$ be a countable amenable group, $(V,\|\cdot\|)$ a complete semi-normed space, $\pi: G \to \mathrm{Isom}(V)$ a uniformly bounded representation, and a set map $\varphi \in \Maps$. If
$$
\inf_{v \in V} \limsup_{F \to G} \left\|\frac{\varphi(F)}{|F|}-A_F v\right\| = 0,
$$
then there exists $v \in V$ such that 
\begin{equation}
\label{eqn:additive_realization}
\lim_{F \to G} \left\|\frac{\varphi(F)}{|F|}-A_F v\right\| = 0.
\end{equation}
\end{restatable}

In order to have a better understanding of the relationship between the complete semi-normed space $V$ and both additive and asymptotically additive set maps, we establish in \S \ref{sec:isom} several isomorphism theorems between adequate linear spaces. In particular, we prove that there is an identification between $V$ and the space of additive maps $\AdMaps$ (see Proposition \ref{prop:iso1}), and also between the quotient spaces $V/\wCob$ and $[\AsMaps]$ (see Theorem \ref{thm:iso4}), where $\wCob$ is the space of \emph{weak coboundaries} and $[\AsMaps]$ is an appropriate space of classes of asymptotically additive set maps.

Through the aforementioned isomorphism results, it is possible, and very important for the applications, to prove a relative version of Theorem \ref{thm:cuneo}. Given $\varphi \in \Maps$ and a subset $W \subseteq V$, we say that $\varphi$ is \emph{asymptotically additive relative to $W$} if
\begin{equation*}
\inf_{w \in W} \limsup_{F \to G} \left\|\frac{\varphi(F)}{|F|}-A_F w\right\| = 0.
\end{equation*}
We call \emph{additive realization} of $\varphi$ any element $v \in V$ that satisfies Equation \eqref{eqn:additive_realization}, and denote by $\witness$ the set of additive realizations of $\varphi$. It is of relevance to determine whether $\witness \cap W$ is empty or not. Denote $\mathcal{L}_W$ the subspace of weak $W$-coboundaries (for precise definitions and properties see \S \ref{sec:wc}). In \S \ref{sec:relative}, we prove the following result.

\begin{restatable}{thm}{theoremtwo}
\label{thm:cuneo-rel}
Let $G$ be a countable amenable group, $(V,\|\cdot\|)$ a complete semi-normed space, $\pi: G \to \mathrm{Isom}(V)$ a uniformly bounded representation, and a set map $\varphi \in \Maps$. Given a non-empty subset $W \subseteq V$, the following are equivalent:
\begin{enumerate}
    \item[(A1)] The set map $\varphi$ is asymptotically additive relative to $W$.
    \item[(A2)] There exists an additive realization $v \in \overline{W+\Cob_W}$ of $\varphi$.
\end{enumerate}
In particular, if $\varphi$ is asymptotically additive relative to $W$, then $\witness \subseteq \overline{W+\Cob}$, and the following dichotomy holds:
\begin{enumerate}
    \item[(B1)] There exists an additive realization $w \in W$ of $\varphi$.
    \item[(B2)] There exists an additive realization $v \in V \setminus (W+\wCob)$ of $\varphi$.
\end{enumerate}
\end{restatable}

Interestingly, in this result we identify the set at which the additive realization belongs. It was an open question that naturally arose from Cuneo's work whether asymptotically additive sequences with certain regularity have an additive realization with the same regularity. That is, when  $W$ is a particular class of regular functions. This result clarifies the situation exhibiting the two possible cases.

The main application of our results is in the realm of non-additive thermodynamic formalism. This is a generalization of classical results on thermodynamic formalism replacing the pressure of a continuous function with the pressure of a sequence of  continuous functions. Falconer \cite{falconer1988} introduced this set of ideas with the purpose of studying dimension theory of non-conformal dynamical systems. A great deal of work has followed and the theory has been greatly expanded (see \cite{barreira2, feng-pablo, iommi-yuki, pesin}). This formalism can be applied to study general products of matrices and other types of cocycles \cite{cao, feng1, feng2, feng-antti}. This was the original purpose of Feng and Huang \cite{feng-huang} to introduce the notion of asymptotically additive sequences. In \S \ref{sec:thermo}, by means of Theorem \ref{thm:cuneo} and Corollary \ref{cor:cuneo}, we develop and study thermodynamic formalism for asymptotically additive set maps in the setting of countable amenable groups. We define the pressure and prove a variational principle extending the theory developed for single functions. Even for the classical case of $\Z$-actions, Theorem \ref{thm:cuneo-rel} clarifies the situation in which it was not known whether asymptotically additive sequences of regular functions (hence, having good thermodynamic properties) have an additive realization with the same regularity.

\section{Preliminaries}
In this section we collect definitions and properties that will be used throughout the article.

\subsection{Amenable groups and complete semi-normed spaces}

Let $G$ be a countable discrete group with identity element $1_G$. We denote by $\FSet$ the set of non-empty finite subsets of $G$. Given $K \in \FSet$ and $\delta > 0$, we say that $F \in \FSet$ is {\bf $(K,\delta)$-invariant} if
$$
|KF \Delta F| \leq \delta|F|,
$$
and say that $G$ is {\bf amenable} if for every $(K,\delta) \in \FSet \times \R_{> 0}$ there exists a $(K,\delta)$-invariant set $F \in \FSet$. A sequence $(F_n)_n$ in $\FSet$ is {\bf (left) F{\o}lner} for $G$ if
$$
\lim_{n \to \infty} \frac{\left|F_n \Delta gF_n\right|}{\left|F_n\right|}=0 \quad \text{for each}~g \in G.
$$
A countable discrete group is amenable if and only if it has a F{\o}lner sequence.

Given a semi-normed vector space $(V,\|\cdot\|)$, we say that a subset $A \subseteq V$ is
\begin{itemize}
    \item {\bf $\|\cdot\|$-closed} if for any sequence $(v_n)_n$ in $A$ and for any $v \in V$ such that $\|v_n - v\| \to 0$, we have that $v \in A$, and
    \item {\bf $\|\cdot\|$-complete} if for any $\|\cdot\|$-Cauchy sequence $(v_n)_n$ in $A$, there exists $v \in A$ such that $\|v_n - v\| \to 0$.
\end{itemize}
A {\bf complete semi-normed vector space} is a semi-normed vector space $(V,\|\cdot\|)$ that is $\|\cdot\|$-complete. It is well-known that a semi-normed vector space is $\|\cdot\|$-complete if and only if it is complete in the sense of nets \cite[\S 6, Theorem 24]{kelley1975}. Given a subspace $U \subseteq V$, we define the {\bf quotient space}
$$
V/U = \{v + U: v \in V\}
$$
and the corresponding quotient map $p: V \to V/U$ given by $p(v) = v + U$. The quotient topology in $V/U$ is given by the semi-norm
$$
\|v + U\|_U = \inf_{u \in U} \|v+u\|,
$$
and the map $p$ is always open and continuous with respect to this topology. Moreover, if $V$ is $\|\cdot\|$-complete and $U$ is a subspace, then $V/U$ is $\|\cdot\|_U$-complete, and the semi-norm $\|\cdot\|_U$ is a norm if and only $U$ is $\|\cdot\|$-closed; as a consequence, if $V$ is $\|\cdot\|$-complete and $U$ is $\|\cdot\|$-closed, the quotient space $V/U$ is a Banach space. See \cite[Proposition 3.1]{swartz2009}.

Given a subset $W \subseteq V$ and a subspace $U$, we denote $[W]_U = \{w + U: w \in W\} \subseteq V/U$. Notice that $[W + U]_U = [W]_U$. If the context is clear, we will omit the subindex $U$ in the bracket.

Throughout this paper, $G$ will be a countable discrete amenable group and $(V,\|\cdot\||)$ a complete semi-normed vector space.

\subsection{Representations and coboundaries}
\label{sec:wc}

Let $\Isom(V)$ denote the group of bounded invertible linear operators on $V$. A {\bf representation} of $G$ on $V$ is a group homomorphism $\pi: G \to \Isom(V)$. We say that $\pi$ is {\bf uniformly bounded} if
$$
C_\pi := \sup_{g \in G}\left\|\pi(g)\right\|_{\mathrm{op}}<\infty,
$$
where $\left\|\cdot\right\|_{\mathrm{op}}$ denotes the linear operator norm. Notice that $C_\pi \geq 1$ since $\left\|\pi(1_G)\right\|_{\mathrm{op}} = \left\|\mathrm{id}_V\right\|_{\mathrm{op}} = 1$. We say $\pi$ is {\bf isometric} if $C_\pi = 1$ or, equivalently, if $\pi(g)$ is an isometry for each $g \in G$.

Given $F \in \FSet$, we define $S_F, A_F: V \to V$, respectively,  by
$$
S_F v = \sum_{g \in F} \pi(g^{-1}) v \quad \text{and} \quad A_F v = \frac{1}{|F|} S_F v   \quad   \text{for}~v \in V. 
$$
Observe that both $S_F$ and $A_F$ are linear, and that
$$
\|A_F v\| \leq  C_\pi\|v\|  \quad   \text{and} \quad  A_{Fg^{-1}} v = \pi(g) A_F v \quad \text{for each}~g \in G.
$$

For a given subset $W \subseteq V$, consider the subspace of {\bf $W$-coboundaries}
$$
\Cob_W := \langle \{w-\pi(g)w: w \in W, g \in G\} \rangle,
$$
where $\langle\cdot\rangle$ denotes the linear span. If $W = V$, we simply write $\Cob$ instead of $\Cob_V$. Notice that $\Cob_W \subseteq \Cob \subseteq \wCob \subseteq V$, where the closure is taken with respect to the semi-norm $\|\cdot\|$. The space $\wCob$ is referred as the closed subspace of {\bf weak coboundaries}. Clearly, 
$$
\|v\| \geq \|v+\Cob_W\|_{\Cob_W} \geq \|v+\Cob\|_{\Cob} = \|v+\wCob\|_{\wCob}  \quad   \text{for every}~v \in V,
$$
where the last equality can be obtained by an standard approximation argument.

\subsection{Set maps and their limits}
\label{sec:limits}

A function of the form $\varphi: \FSet \to V$ will be called a {\bf set map}. Given $L \in V$, we write
$$
\lim_{F \to G} \varphi(F) = L
$$
if for every $\epsilon > 0$, there exists $(K,\delta) \in \FSet \times \R_{>0}$ such that for every $(K, \delta)$-invariant set $F \in \FSet$ we have $\|\varphi(F)-L\|<\epsilon$. In this case, we say that {\bf $\varphi(F)$ converges to $L$ as $F$ becomes more and more invariant}. Notice that $\FSet \times \R_{>0}$ is a directed set for the partial order $\preccurlyeq$ given by $(K,\delta) \preccurlyeq (K',\delta')$ if and only if $K \subseteq K'$ and $\delta \geq \delta'$. 

Notice that $\lim_{F \to G} \varphi(F) = L$ if and only if $\lim_{F \to G} \|\varphi(F) - L\| = 0$. For a real-valued set map $\varphi: \FSet \to \R$, we define
$$
\limsup_{F \to G} \varphi(F) := \inf_{(K,\delta)} \sup\{\varphi(F): F \textrm{ is } (K,\delta)\textrm{-invariant}\}
$$
and
$$
\liminf_{F \to G} \varphi(F) := \sup_{(K,\delta)} \inf\{\varphi(F): F \textrm{ is } (K,\delta)\textrm{-invariant}\}.
$$
Observe that, since $\FSet \times \R_{>0}$ is directed, $\liminf_{F \to G} \varphi(F) \leq \limsup_{F \to G} \varphi(F)$.

\begin{lemma}
\label{lem:limsup}
Given a set map $\varphi: \FSet \to \R$ and $L \in \R$, the following are equivalent.
\begin{enumerate}
    \item[(i)] $\lim_{F \to G} \varphi(F) = L$.
    \item[(ii)] $\lim_{n \to \infty} \varphi(F_n) = L$ for every F{\o}lner sequence $(F_n)_n$.
    \item[(iii)] $\liminf_{F \to G} \varphi(F) = L = \limsup_{F \to G} \varphi(F)$.
\end{enumerate}
\end{lemma}

\begin{proof}
Pick a F{\o}lner sequence $(F_n)_n$ and $\epsilon > 0$. If $\lim_{F \to G} \varphi(F) = L$, then there exists $(K,\delta)$ such that $|\varphi(F)-L| < \epsilon$ for every $(K,\delta)$-invariant set $F$. Pick $n_0$ such that $F_n$ is $(K,\delta)$-invariant for every $n \geq n_0$. Then, $|\varphi(F_n)-L| < \epsilon$ for every $n \geq n_0$, so $\lim_{n \to \infty} \varphi(F_n) = L$ and we have that $(i) \implies (ii)$.

Suppose now that $\lim_{n \to \infty} \varphi(F_n) = L$ for every F{\o}lner sequence $(F_n)_n$. Pick a cofinal sequence $(K_n,\delta_n)$ and, for each $n$, pick $(K_n,\delta_n)$-invariant sets $F^-_n$ and $F^+_n$ such that
\begin{align*}
\varphi(F^-_n) - 2^{-n} &   \leq \inf\{\varphi(F): F \textrm{ is } (K_n,\delta_n)\textrm{-invariant}\}  \\
                        &   \leq \sup\{\varphi(F): F \textrm{ is } (K_n,\delta_n)\textrm{-invariant}\} \leq \varphi(F^+_n) + 2^{-n}.
\end{align*}
Taking the limit in $n$, we obtain that
\begin{align*}
\lim_{n \to \infty} \varphi(F^-_n) & \leq \lim_{n \to \infty} \inf\{\varphi(F): F \textrm{ is } (K_n,\delta_n)\textrm{-invariant}\}  \\
                        &   \leq \lim_{n \to \infty} \sup\{\varphi(F): F \textrm{ is } (K_n,\delta_n)\textrm{-invariant}\} \leq \lim_{n \to \infty} \varphi(F^+_n),
\end{align*}
and since $(F^-_n)_n$ and $(F^+_n)_n$ are F{\o}lner sequences and $(K_n,\delta_n)$ is cofinal, we conclude that
$$
L \leq \liminf_{F \to G} \varphi(F) \leq \limsup_{F \to G} \varphi(F) \leq L,
$$
so $(ii) \implies (iii)$.

Finally, suppose that $\liminf_{F \to G} \varphi(F) = L = \limsup_{F \to G} \varphi(F)$. Let $\epsilon > 0$ and pick $(K,\delta) \in \FSet \times \R_{> 0}$ such that
$$
\sup\{\varphi(F): F \textrm{ is } (K,\delta)\textrm{-invariant}\} - \epsilon \leq L \leq \inf\{\varphi(F): F \textrm{ is } (K,\delta)\textrm{-invariant}\} + \epsilon.
$$
This implies that, for every $(K,\delta)$-invariant set $F$, $\varphi(F) - \epsilon \leq L \leq \varphi(F) + \epsilon$. In other words, $|\varphi(F)-L| < \epsilon$, i.e., $\lim_{F \to G} \varphi(F) = L$, and we have that $(iii) \implies (i)$.
\end{proof}

\section{A realization theorem for asymptotically additive set maps} \label{sec:space-maps}

Fix a countable amenable group $G$, a complete semi-normed space $(V,\|\cdot\|)$, and a uniformly bounded representation $\pi: G \to \mathrm{Isom}(V)$. Notice that there are natural left actions $G \acts V$ and $G \acts \FSet$ given by $g \cdot v = \pi(g)v$ and $g \cdot F = Fg^{-1}$, respectively. Given a set map $\varphi: \FSet \to V$, we define
$$
\vertsup{\varphi} := \sup_{F \in \FSet}\left\|\frac{\varphi(F)}{|F|}\right\| \quad \text{and} \quad 
\vertG{\varphi} := \limsup_{F \to G}\left\|\frac{\varphi(F)}{|F|}\right\|,
$$
and say that $\varphi$ is {\bf bounded} if $\vertsup{\varphi} < \infty$. A map $\varphi$ is {\bf $G$-equivariant} if
$$
g \cdot \varphi(F) = \varphi(g \cdot F) \quad \text{for every $g \in G$ and $F \in \FSet$}.
$$
Notice that $\vertsup{\cdot}$ and $\vertG{\cdot}$ are semi-norms. We define the {\bf space of bounded and $G$-equivariant set maps} $\Maps$ as
$$
\Maps := \{\varphi: \FSet \to V \mid \vertsup{\varphi} < \infty \textrm{ and } \varphi \textrm{ is $G$-equivariant}\}.
$$
Observe that $\|\varphi(g \cdot F)\| \leq \vertsup{\varphi}|F|$ for every $\varphi \in \Maps$, $g \in G$, and $F \in \FSet$.

\begin{proposition}
If $(V,\|\cdot\|)$ is a Banach space (resp. complete semi-normed space), then $(\Maps,\vertsup{\cdot})$ is a Banach space (resp. complete semi-normed space).
\end{proposition}

\begin{proof}
It is direct that $\vertsup{\cdot}$ is a norm (resp. semi-norm) if and only if $\|\cdot\|$ is a norm (resp. semi-norm). Let $(\varphi_n)_n$ be a $\vertsup{\cdot}$-Cauchy sequence in $\Maps$. This implies that $(\varphi_n(F))_n$ is a $\|\cdot\|$-Cauchy sequence in $V$ for every $F \in \FSet$. Since $V$ is complete, for every $F \in \FSet$, there exists a limit $\varphi(F) := \lim_n \varphi_n(F)$ and we can define a set map $\varphi: \FSet \to V$ pointwise. Then, $\varphi$ belongs to $\Maps$. Indeed, $\vertsup{\varphi} < \infty$ because the convergence is uniform, and $\varphi$ is $G$-equivariant because, for every $g \in G$ and $F \in \FSet$,
$$
g \cdot \varphi(F) = \pi(g) \lim_n \varphi_n(F) = \lim_n \pi(g) \varphi_n(F) = \lim_n \varphi_n(g \cdot F) = \varphi(g \cdot F),
$$
where we have used that $\pi(g)$ is continuous.
\end{proof}

\subsection{Averages of weak coboundaries}

 Fix a semi-normed vector space $(V,\|\cdot\|)$ and a uniformly bounded representation $\pi: G \to \Isom(V)$.

\begin{lemma}
\label{lem:cob}
If $u \in \wCob$, then $\limsup_{F \to G} \left\|A_F u\right\| = 0$.
\end{lemma}

\begin{proof}
For every $\epsilon>0$, there exist $v_1,\dots,v_\ell \in V$, $\lambda_1,\dots,\lambda_\ell \in \R$, and $g_1,\dots,g_\ell \in G$ such that
$$
\left\|u - \sum_{i=1}^\ell \lambda_i(v_i - \pi(g_i) v_i)\right\| < \epsilon.
$$
Since $A_F$ is linear and $\left\|A_F v\right\| \leq C_\pi\left\|v\right\|$ for every $v \in V$, it follows that
\begin{align*}
\left\|A_F u\right\|    &   \leq    \left\|A_F\left(u-\sum_{i=1}^\ell \lambda_i(v_i-\pi(g_i)v_i)\right)\right\| + \left\|A_F\left(\sum_{i=1}^\ell \lambda_i(v_i-\pi(g_i) v_i)\right)\right\| \\
                        &   =    \left\|A_F\left(u-\sum_{i=1}^\ell \lambda_i(v_i-\pi(g_i)v_i)\right)\right\| + \left\|\sum_{i=1}^\ell \lambda_i A_F(v_i-\pi(g_i) v_i)\right\| \\
                        &   \leq    C_\pi\left\|u-\sum_{i=1}^\ell \lambda_i(v_i-\pi(g_i) v_i)\right\| + \sum_{i=1}^\ell |\lambda_i| \left\|A_F(v_i-\pi(g_i) v_i)\right\| \\
                        &   \leq    C_\pi\epsilon + \sum_{i=1}^\ell |\lambda_i|\frac{1}{|F|}\left\|S_Fv_i -  S_F\pi(g_i)v_i\right\| \\
                        &   =    C_\pi\epsilon + \sum_{i=1}^\ell |\lambda_i|\frac{1}{|F|}\left\|\sum_{g \in F \setminus g_i^{-1}F} \pi(g^{-1})v_i -  \sum_{g \in g_i^{-1}F \setminus F} \pi(g^{-1})v_i\right\| \\
                        &   \leq    C_\pi\epsilon + \sum_{i=1}^\ell |\lambda_i|\frac{|g_i^{-1}F \Delta F|}{|F|}C_\pi\|v_i\|.
\end{align*}
Taking the limit as $F \to G$ and then using that $\epsilon$ is arbitrary yields the result.
\end{proof}

\begin{proposition}
\label{prop:cob1}
Let $U$ be a subspace of $V$. If $U \subseteq \wCob$, then,
$$\limsup_{F \to G} \left\|A_F v\right\| \leq C_\pi\|v+U\|_U \quad \text{for every}~v \in V.
$$
\end{proposition}

\begin{proof}
Let $v \in V$. For every $\epsilon>0$, there exists $u \in U$ such that $\left\|v + u\right\| \leq \|v+U\|_U + \epsilon$. It follows that
$$
\left\|A_F v\right\| \leq \left\|A_F(v+u)\right\|+\left\|A_Fu\right\| \leq  C_\pi(\|v+U\|_U + \epsilon) + \left\|A_Fu\right\|.
$$
By Lemma \ref{lem:cob}, taking the limit as $F \to G$ and using that $\epsilon$ is arbitrary yields the result.
\end{proof}

\begin{proposition}
\label{prop:cob2}
Let $U$ be a subspace of $V$. If $W \subseteq V$ is a subset such that $\Cob_W \subseteq U$, then
$$
\|w+U\|_U \leq \inf_{F \in \FSet}\left\|A_F w\right\| \quad \text{for every}~w \in \overline{\langle W \rangle}.
$$
In particular,
$$
\|v+\Cob\|_\Cob \leq \inf_{F \in \FSet}\left\|A_F v\right\| \quad \text{for every}~v \in V.
$$
\end{proposition}

\begin{proof}
If $w \in W$, observe that $A_F w \in w + U$ for each finite $F \subseteq G$, since
$$
w-A_F w = \frac{1}{|F|} \sum_{g \in F}(w - \pi(g^{-1})w) \in \Cob_W.
$$
Similarly, if $w \in \langle W \rangle $, the linearity of $A_F$ shows that $w-A_F w  \in \Cob_W$, so $A_F w \in w + U$. Then, using the definition of $\|\cdot\|_U$, we obtain that $\left\|A_F w\right\| \geq \|w+U\|_U$. Taking the infimum over $F$, we obtain that
$$
\|w+U\|_U \leq \inf_{F \in \FSet}\left\|A_F w\right\| \quad \text{for every}~w \in \langle W \rangle.
$$
Finally, if $w \in \overline{\langle W \rangle}$, then for any $\epsilon > 0$, there exists $w_\epsilon \in \langle  W \rangle$ such that $\|w - w_\epsilon\| \leq \epsilon$ and
\begin{align*}
\|w + U\|_U &   \leq \|(w-w_\epsilon) + U\|_U + \|w_\epsilon + U\|_U  \\
            &   \leq \|w-w_\epsilon\| + \inf_{F \in \FSet}\left\|A_F w_\epsilon\right\|    \\
            &   \leq \epsilon + \inf_{F \in \FSet}\left(\left\|A_F w\right\| + \left\|A_F (w-w_\epsilon)\right\|\right)    \\
            &   \leq \epsilon + \inf_{F \in \FSet}\left(\left\|A_F w\right\| + C_\pi\left\|w_\epsilon-w\right\|\right)    \\
            &   \leq \epsilon + \inf_{F \in \FSet}\left(\left\|A_F w\right\| + C_\pi\epsilon\right)    \\
            &   =    (1+C_\pi)\epsilon + \inf_{F \in \FSet}\left\|A_F w\right\|,  
\end{align*}
and since $\epsilon$ was arbitrary we conclude. The last statement follows from considering the special case $W=V$ and $U = \Cob$.
\end{proof}

\begin{corollary}
\label{cor:cob}
Given $v \in V$,  we have that $v \in \wCob$ if and only if $\limsup_{F \to G} \left\|A_F v\right\| = 0$.
\end{corollary}

\begin{proof}
If $v \in \wCob$, then $\limsup_{F \to G} \left\|A_F v\right\| = 0$ by Lemma \ref{lem:cob}. It remains to prove the converse. Let $v \in V$ be such that $\limsup_{F \to G} \left\|A_F v\right\| = 0$. Due to Proposition \ref{prop:cob2}, it follows that
$$
\|v+\wCob\|_{\wCob} \leq \|v+\Cob\|_{\Cob} \leq \inf_{F \in \FSet} \left\|A_F v\right\| \leq \limsup_{F \to G} \left\|A_F v\right\| = 0,
$$
so $\|v+\wCob\|_{\wCob} = 0$. Therefore, $v+\wCob = \wCob$, i.e., $v \in \wCob$.
\end{proof}

\begin{corollary}
\label{cor:LW}
If $U$ is a subspace of $V$ and $W \subseteq V$ is a subset such that $\Cob_W \subseteq U \subseteq \wCob$, then
$$
C_\pi^{-1}\limsup_{F \to G}\left\|A_F w\right\| \leq  \|w+U\|_U \leq \inf_{F \in \FSet}\left\|A_F w\right\| \leq \limsup_{F \to G}\left\|A_F w\right\| \quad   \text{for every}~w \in \overline{\langle W \rangle}.
$$
If, in addition, $\pi$ is isometric, then
$$
\|w+U\|_U = \limsup_{F \to G} \left\|A_F w\right\| =  \inf_{F \in \FSet} \left\|A_F w\right\|  \quad   \text{for every}~w \in \overline{\langle W \rangle}.
$$
\end{corollary}

\begin{proof}
This is direct from Proposition \ref{prop:cob1}, Proposition \ref{prop:cob2}, and recalling that $\pi$ being isometric means that $C_\pi = 1$.
\end{proof}

\subsection{Asymptotically additive set maps}

A set map $\varphi \in \Maps$ will be called {\bf additive} if
$$
\varphi(E \sqcup F) = \varphi(E) + \varphi(F)
$$
for every pair of disjoint sets $E,F \in \FSet$ or, equivalently, if
$$
\varphi(F) = \sum_{E \in \Part} \varphi(E)
$$
for every $F \in \FSet$ and for every partition $\Part$ of $F$. We will denote by $\AdMaps$ the set of additive set maps in $\Maps$. 

\begin{proposition}
The set $\AdMaps$ is a $\vertsup{\cdot}$-closed subspace of $\Maps$.
\end{proposition}

\begin{proof}
Clearly, if $\varphi,\psi \in \AdMaps$ and $\alpha \in \R$, then $\varphi + \alpha\psi \in \AdMaps$, so $\AdMaps$ is a subspace.

If $(\varphi_n)_n$ is a sequence in $\AdMaps$ $\vertsup{\cdot}$-converging to $\varphi \in \Maps$, then $\varphi(F) = \lim_{n \to \infty} \varphi_n(F)$ for every $F \in \FSet$. In particular, we have that
\begin{align*}
\varphi(E \sqcup F) = \lim_{n \to \infty} \varphi_n(E \sqcup F) = \lim_{n \to \infty} \varphi_n(E) + \lim_{n \to \infty} \varphi_n(F) = \varphi(E) + \varphi(F)
\end{align*}
for every pair of disjoint sets $E,F \in \FSet$, so $\varphi \in \AdMaps$.
\end{proof}

Even though the subspace $\AdMaps$ is $\vertsup{\cdot}$-closed, it is of interest to study its $\vertG{\cdot}$-closure, that it usually induces a strictly larger space. A set map $\varphi \in \Maps$ will be called {\bf asymptotically additive} if
$$
\inf_{\psi \in \AdMaps} \vertG{\varphi - \psi} = 0,
$$
i.e., if it belongs to the $\vertG{\cdot}$-closure of $\AdMaps$. We denote by $\AsMaps$ the set of asymptotically additive maps in $\Maps$. It turns out that every asymptotically additive set map is asymptotically indistinguishable from an additive map, as established by the following result.

\theoremone*

In other words, Theorem \ref{thm:cuneo} establishes that for every $\varphi \in \AsMaps$, there exists $\psi \in \AdMaps$ such that $\vertG{\varphi-\psi}=0$.

\begin{proof}
Suppose that $\varphi \in \Maps$ is asymptotically additive. Then, for each $n \in \N$, there exists $\psi_n \in \AdMaps$ such that
$$
\vertG{\varphi-\psi_n} \leq 2^{-n}
$$
or, equivalently,
$$
\limsup_{F \to G}\left\|\frac{\varphi(F)}{|F|}-A_Fv_n\right\| \leq 2^{-n}
$$
for $v_n = \psi_n(\{1_G\})$. Due to Proposition \ref{prop:cob2}, for all $m,n \in \N$,
\begin{align*}
\left\|(v_n+\wCob)-(v_m+\wCob)\right\|_{\wCob} &   \leq  \limsup_{F \to G} \left\|A_Fv_n -A_Fv_m\right\|   \\
     &   \leq \limsup_{F \to G} \left\|A_F v_n - \frac{\varphi(F)}{|F|}\right\| + \limsup_{F \to G} \left\|\frac{\varphi(F)}{|F|} - A_F v_m\right\|   \\
     &   \leq    2^{-n} + 2^{-m},
\end{align*}
so $(v_n+\wCob)_n$ is as Cauchy sequence in $V/\wCob$. Since $V/\wCob$ is complete \cite[Theorem 4.8.7]{narici2011}, there exists $v \in V$ such that $\lim_n \|(v+\wCob)-(v_n+\wCob)\|_{\wCob}=0$. By Proposition \ref{prop:cob1} (see also Corollary \ref{cor:LW}),
\begin{align*}
\limsup_{F \to G}\left\| \frac{\varphi(F)}{|F|}-A_F v\right\|   &   \leq    \limsup_{F \to G} \left\|\frac{\varphi(F)}{|F|}-A_F v_n\right\| + \limsup_{F \to G} \left\|A_F(v_n-v)\right\| \\
&   \leq    2^{-n} + C_\pi\left\|(v+\wCob)-(v_n+\wCob)\right\|_{\wCob} \to 0
\end{align*}
as $n \to \infty$. Thus, we conclude that $\psi = S(v)$ is such that $\vertG{\varphi-\psi} = 0$.
\end{proof}

\corollaryone*

Any element $v \in V$ that satisfies Equation (\ref{eqn:additive_realization}) will be called an {\bf additive realization} of $\varphi$ and the set of additive realizations of $\varphi$ will be denoted by $\witness$. The next lemma describes the latter set.

\begin{lemma}
\label{lem:witness}
If $v \in V$ is an additive realization of $\varphi$, then $\witness = v + \wCob$.
\end{lemma}

\begin{proof}
If $u \in \wCob$, then, by Lemma \ref{lem:cob},
$$
\limsup_{F \to G} \left\|\frac{\varphi(F)}{|F|} - A_F(v+u)\right\| \leq \limsup_{F \to G} \left\|\frac{\varphi(F)}{|F|} - A_Fv\right\| + \limsup_{F \to G} \left\|A_Fu\right\| = 0,
$$
so, $v + \wCob \subseteq \witness$. If $v' \in \witness$ is not in $v + \wCob$, then $v' + \wCob \neq v + \wCob$, so
$$
0 < \|(v' + \wCob) - (v + \wCob)\|_{\wCob} \leq \limsup_{F \to G} \|A_Fv' - A_Fv\|.
$$
Therefore,
\begin{align*}
0   &   <  \limsup_{F \to G} \left\|A_Fv' - \frac{\varphi(F)}{|F|}\right\| + \limsup_{F \to G} \left\|\frac{\varphi(F)}{|F|} - A_Fv\right\| = \limsup_{F \to G} \left\|A_Fv' - \frac{\varphi(F)}{|F|}\right\|,
\end{align*}
so $v'$ is not an additive realization of $\varphi$.
\end{proof}

The notion of asymptotical additivity generalizes the one given by Feng and Huang \cite{feng-huang} for $\Z$-actions, the Koopman representation, and the Banach space of continuous functions (see \S \ref{sec:folner} for details). In that setting, Cuneo \cite[Theorem 1.2]{cuneo2020additive} proved a version of Corollary \ref{cor:cuneo}. Therefore, our results extend this framework to a much more general one, both in terms of the groups acting and the representations involved.

\subsection{Asymptotical additivity in the sequential case}
\label{sec:folner}

In many situations, it is common to work with families of maps $\varphi$ that are partially defined. For instance, given a F{\o}lner sequence $(F_n)_n$, we may be interested in studying limits only along it, such as
$$
\limsup_{n \to \infty}\left\|\frac{\varphi(F_n)}{|F_n|}\right\|.
$$
Here we show how our formalism is a natural and more encompassing way of studying this kind of situations.

Let $\mathcal{S}$ be a $G$-invariant subset of $\FSet$ (i.e., $G \cdot \mathcal{S} \subseteq \mathcal{S}$) such that for every $(K,\delta) \in \FSet \times \R_{> 0}$ there exists a $(K,\delta)$-invariant set $F \in \mathcal{S}$. If $(F_n)_n$ is a F{\o}lner sequence, the natural $\mathcal{S}$ to consider is the $G$-orbit of $(F_n)_n$, namely 
$$
\mathcal{S} = \{g \cdot F_n \mid g \in G, n \in \N\}.
$$
Then, any $\varphi: \{F_n\}_n \to V$ is naturally extended to $\varphi: \mathcal{S} \to V$ by inducing $G$-equivariance, this is to say,
$$
\varphi(g \cdot F_n) := g \cdot \varphi(F_n) = \pi(g) \varphi(F_n).
$$
This suggests to define, for any $G$-invariant subset $\mathcal{S}$, the {\bf set of $\mathcal{S}$-maps} as
$$
\Maps(\mathcal{S}) := \{\varphi: \mathcal{S} \to V \mid \vertiii{\varphi}_{\mathcal{S}} < \infty \textrm{ and } \varphi \textrm{ is $G$-equivariant}\},
$$
where
$$
\vertiii{\varphi}_{\mathcal{S}} := \sup_{F \in \mathcal{S}}\left\|\frac{\varphi(F)}{|F|}\right\|.
$$
Similarly, we can extend the definitions of $\lim_{F \to G}$, $\limsup_{F \to G}$, etc., by restricting all of them to sets $F \in \mathcal{S}$. In particular, asymptotical additivity takes the form
$$
\inf_{v \in V} \limsup_{\substack{F \to G\\F \in \mathcal{S}}} \left\|\frac{\varphi(F)}{|F|} - A_Fv\right\| = 0.
$$

In the case $G=\Z$, a recurrent instance is to deal with sequences $(f_n)_n$ of continuous functions $f_n: X \to \R$, where $X$ is some compact metric space and $T: X \to X$ is a homeomorphism. If we consider $\mathcal{S} = \{[m,n) \cap \Z: m < n\}$, $V = \mathcal{C}(X)$ the space of real-valued continuous functions with the uniform norm, and $\pi$ to be the Koopman representation $\pi: \Z \to \mathcal{C}(X)$ given by $\pi(-n)(f)(x) =  f(T^n x)$, we can see that this situation fits in our context and that our formalism generalizes it. Indeed, in the literature (see \cite[Chapter 7]{barreira3}, \cite{feng-huang}, \cite{holanda}, and \cite{zhao2011asymptotically}), asymptotical additivity is usually defined in this context as
$$
\inf_{f \in \mathcal{C}(X)} \limsup_{n \to \infty} \frac{1}{n}\left\|f_n - S_nf\right\|_{\infty} = 0, 
$$
where $S_nf = \sum_{i=0}^{n-1} f \circ T^i$. Notice that this coincides with the definition of asymptotical additivity given here, because
$$
\frac{1}{n}\left\|f_n - S_nf\right\|_{\infty} = \left\|\frac{\varphi([0,n))}{n} - A_{[0,n)} f\right\|_{\infty},
$$
where $\varphi([m,n) \cap \Z) = f_n \circ T^m$ for every $m < n$, and
\begin{align*}
0	&	=	\inf_{v \in V} \limsup_{\substack{F \to \Z\\F \in \mathcal{S}}} \left\|\frac{\varphi(F)}{|F|} - A_Fv\right\|_{\infty}	\\
	&	=	\inf_{f \in \mathcal{C}(X)} \limsup_{n \to \infty} \sup_{k \in \Z} \left\|\frac{\varphi([0,n) - k)}{n} - A_{[0,n)-k} f\right\|_{\infty}	\\
	&	=	\inf_{f \in \mathcal{C}(X)} \limsup_{n \to \infty} \sup_{k \in \Z} \left\|\pi(k)\frac{\varphi([0,n))}{n} - \pi(k)A_{[0,n)} f\right\|_{\infty}	\\
	&	\leq	\inf_{f \in \mathcal{C}(X)} \limsup_{n \to \infty} \sup_{k \in \Z} \|\pi(k)\|_{\mathrm{op}}\left\|\frac{\varphi([0,n))}{n} - A_{[0,n)} f\right\|_{\infty}	\\
	&	\leq	C_\pi \inf_{f \in \mathcal{C}(X)} \limsup_{n \to \infty}  \left\|\frac{\varphi([0,n))}{n} - A_{[0,n)} f\right\|_{\infty}	\\
	&	\leq	C_\pi \inf_{f \in \mathcal{C}(X)} \limsup_{n \to \infty} \sup_{k \in \Z} \left\|\frac{\varphi([0,n) - k)}{n} - A_{[0,n)-k} f\right\|_{\infty}	= 0,
\end{align*}
where the interval $[0,n)$ is understood as an interval in $\Z$.

\section{Isomorphism theorems}
\label{sec:isom}

We already established that every given asymptotically additive set map admits an additive realization, which already provides information regarding it. However, in many applications, it is of relevance to have finer information regarding the possible additive realizations. In order to better understand this problem, we study natural linear spaces that provide a framework that clarifies the interaction between asymptotically additive set maps and their additive realizations via appropriate quotients and isomorphisms of these spaces.

Let $\NMaps$ be the {\bf kernel} of $\vertG{\cdot}$, that is,
$$
\NMaps = \{\varphi \in \Maps: \vertG{\varphi} = 0\}.
$$
Note that, as a kernel, $\NMaps $ is a $\vertG{\cdot}$-closed and $\vertG{\cdot}$-complete subspace of $\Maps$. We have that a subspace $\mathcal{U}$ is $\vertG{\cdot}$-closed only if $\NMaps \subseteq \mathcal{U}$.  In particular, $\Maps$ is $\vertG{\cdot}$-complete if and only if $\Maps/\NMaps$ is a Banach space. See \cite[Proposition 3.2]{swartz2009} for further details.

\subsection{A characterization of the space of additive set maps}

Let $\AdNMaps$ be the space of additive maps in the kernel of $\vertG{\cdot}$, i.e., $\NMaps \cap \AdMaps$. This space is $\vertsup{\cdot}$-closed since it is the intersection of two $\vertsup{\cdot}$-closed subspaces. Consider the map $S: V \to \Maps$ given by $S(v)(F) = S_Fv$. 

\begin{proposition}
\label{prop:iso1}
The map $S$ is a linear isomorphism between $(V,\|\cdot\|)$ and $(\AdMaps,\vertsup{\cdot})$. Moreover, if $\pi$ is isometric, then $S$ is an isometry.
\end{proposition}

\begin{proof}
Notice that $S$ is injective, linear, and continuous since
$$
S(v_1) = S(v_2) \implies v_1 = S(v_1)(\{1_G\}) = S(v_2)(\{1_G\}) = v_2,
$$
$$
S(v_1 + \alpha v_2)(F) = S_F(v_1 + \alpha v_2) = S_Fv_1 + \alpha S_Fv_2 = S(v_1)(F) + \alpha S(v_2)(F),
$$
and
$$
\vertsup{S(v)} = \sup_{F \in \FSet}\|A_F v\| \leq \sup_{F \in \FSet} C_\pi \|v\| = C_\pi \|v\|.
$$
Moreover, $S(v) \in \AdMaps$ for every $v \in V$, since
$$
S(v)(E \sqcup F) = S_{E \sqcup F}v = S_Ev + S_Fv = S(v)(E) + S(v)(F).
$$

For $\psi \in \AdMaps$, we have that
$$
\psi(F) = \sum_{g \in F} \psi(\{g\}) = \sum_{g \in F} \pi(g^{-1}) \psi(\{1_G\}) = S_F \psi(\{1_G\}) = S(\psi(\{1_G\}))(F), 
$$
so $S(\psi(\{1_G\})) = \psi$, and then $S(V) = \AdMaps$. In particular, the inverse $S^{-1}: \AdMaps \to V$ given by $S^{-1}(\psi) = \psi(\{1_G\})$ is also continuous, since
$$
\|S^{-1}(\psi)\| = \|\psi(\{1_G\})\| \leq \vertsup{\psi}.
$$
In other words, $S$ is a linear isomorphism between $V$ and $\AdMaps$. Moreover, if $\pi$ is isometric (i.e., if $C_\pi = 1$), $S$ is an isometry, since
$$
\|v\| = \|S^{-1}(S(v))\| \leq \vertsup{S(v)} = \sup_{F \in \FSet} \frac{\|S(v)(F)\|}{|F|} = \sup_{F \in \FSet}\|A_F v\| \leq \|v\|.
$$
\end{proof}

From the identification between $(V,\|\cdot\|)$ and $(\AdMaps,\vertsup{\cdot})$, we can see that
$$
\psi \in \AdNMaps \iff \limsup_{F \to G} \|A_F \psi(\{1_G\})\| = 0 \iff S^{-1}(\psi) \in \wCob.
$$
In other words, $S^{-1}(\NMaps \cap \AdMaps) = \wCob$ or, equivalently, $S(\wCob) = \AdNMaps$. Also, observe that $(V/\wCob,\|\cdot\|_{\wCob})$ is a Banach space because $\wCob$ is closed in $V$. Combining these observations, we obtain the following.

\begin{corollary}
\label{cor:iso2}
The restriction $S:\wCob \to \AdNMaps$ is a linear isomorphism between the subspaces $(\wCob,\|\cdot\|)$ and $(\AdNMaps,\vertsup{\cdot})$. Moreover, the linear transformation given by $\overline{S}(v + \wCob) = S(v) + \AdNMaps$ is a linear isomorphism between $(V/\wCob,\|\cdot\|_{\wCob})$ and $(\AdMaps/\AdNMaps,\vertiii{\cdot}_{\AdNMaps})$, where $\vertiii{\cdot}_{\AdNMaps}$ is the quotient norm. In particular, $(\AdMaps/\AdNMaps,\vertiii{\cdot}_{\AdNMaps})$ is a Banach space.
\end{corollary}

\begin{proof}
First, notice that the restriction $S:\wCob \to \AdNMaps$ is well-defined, since if $u \in \wCob$, then $S(u) \in \AdMaps$ and, due to Corollary \ref{cor:cob},
$$
\vertG{S(u)} = \limsup_{F \to G} \left\|\frac{S(u)(F)}{|F|}\right\| = \limsup_{F \to G} \|A_F(u)\| = 0,
$$
so $S(u) \in \AdNMaps$. Since $S$ is a restriction of a linear isomorphism, it suffices to prove that $S^{-1}(\psi)$ belongs to $\wCob$ for any $\psi \in \AdNMaps$. Indeed, 
$$
\limsup_{F \to G}\|A_F S^{-1}(\psi)\| = \limsup_{F \to G}\|A_F \psi(\{1_G\})\| = \limsup_{F \to G}\left\|\frac{\psi(F)}{|F|}\right\| = \vertG{\psi} = 0,
$$
and, again by Corollary \ref{cor:cob}, we conclude that $S^{-1}(\psi) \in \wCob$, so $(\wCob,\|\cdot\|)$ and $(\AdNMaps,\vertsup{\cdot})$ are linearly isomorphic.

Finally, by Proposition \ref{prop:iso1}, the map $S$ is a linear isomorphism between $(V,\|\cdot\|)$ and $(\AdMaps,\vertsup{\cdot})$. Since the restriction of $S$ is also a linear isomorphism between the corresponding subspaces $(\wCob,\|\cdot\|)$ and $(\AdNMaps,\vertsup{\cdot})$, we obtain that the map $\overline{S}$ induced by $S$ is a linear isomorphism between the corresponding quotient spaces $(V/\wCob,\|\cdot\|_{\wCob})$ and $(\AdMaps/\AdNMaps,\vertiii{\cdot}_{\AdNMaps})$, and since the former is a Banach space, we conclude that so is the latter.
\end{proof}

\begin{remark}
Notice that all the linear transformations involved in Corollary \ref{cor:iso2} are isometries if we assume that $\pi$ is isometric.
\end{remark}

\subsection{Asymptotical completeness of the space of set maps}

We now prove that the space of set maps is complete in an asymptotical sense, that is, with respect to the semi-norm $\vertG{\cdot}$.

\begin{proposition}
\label{prop:cauchy}

The space $\Maps$ is $\vertG{\cdot}$-complete. More specifically, if $(\varphi_m)_m$ is a $\vertG{\cdot}$-Cauchy sequence in $\Maps$, then there exist a subsequence $(\varphi_n)_n$ of $(\varphi_m)_m$ and an increasing cofinal sequence $((K_n,\delta_n))_n$ in $\FSet \times \R_{>0}$ such that
$$
\vertG{\varphi - \varphi_n} \to 0,
$$
where
$$
\varphi = \sum_n \mathds{1}_{\mathcal{I}_n}\varphi_n \in \Maps
$$
and $(\mathcal{I}_n)_n$ is the partition of $\FSet$ given by
$$
\mathcal{I}_n = \{F \in \FSet: F \textrm{ is $(K_n,\delta_n)$-invariant but not $(K_{n+1},\delta_{n+1})$-invariant}\}.
$$
\end{proposition}

\begin{proof}
Let $(\varphi_m)_m$ be a $\vertG{\cdot}$-Cauchy sequence in $\Maps$. Then, there exists a constant $C > 0$ such that $\vertG{\varphi_m} \leq C$ for all $m$. Consider a subsequence $(\varphi_n)_n$ such that
$$
\vertG{\varphi_{n+1} - \varphi_n} \leq 2^{-n} \quad \text{for every}~n.
$$

Observe that
$$
\vertG{\varphi_n} = \inf_{(K,\delta)} \sup \left\{\left\|\frac{\varphi_n(F)}{|F|}\right\|: F \textrm{ is $(K,\delta)$-invariant}\right\}
$$
and let $(K_0,\delta_0) = (1_G,1)$. For $n \geq 1$, inductively construct a cofinal sequence $((K_n,\delta_n))_n$ in $\FSet \times \R_{>0}$ such that $K_{n+1} \subsetneq K_n$, $\delta_{n+1} < \delta_n$,
\begin{enumerate}
\item $\sup \left\{\left\|\frac{\varphi_n(F)}{|F|}\right\|: F \textrm{ is $(K_n,\delta_n)$-invariant}\right\} \leq \vertG{\varphi_n} + 2^{-n}$, and
\item $\sup \left\{\left\|\frac{\varphi_{n+1}(F) - \varphi_n(F)}{|F|}\right\|: F \textrm{ is $(K_n,\delta_n)$-invariant}\right\} \leq \vertG{\varphi_{n+1} - \varphi_n} + 2^{-n}$.
\end{enumerate}

The sequence $((K_n,\delta_n))_n$ defines a partition $(\mathcal{I}_n)_n$ of $\FSet$, where $\mathcal{I}_n$ is the set of subsets $F \in \FSet$ that are $(K_n,\delta_n)$-invariant and not $(K_{n+1},\delta_{n+1})$-invariant. Considering this, define $\varphi: \FSet \to V$ given by
$$
\varphi = \sum_n \mathds{1}_{\mathcal{I}_n}\varphi_n.
$$
Notice that $\|\varphi(F)\| = \sum_n \mathds{1}_{\mathcal{I}_n}(F)\|\varphi_n(F)\|$ and $\varphi$ is bounded since
\begin{align*}
\vertsup{\varphi}   &   =       \sup_{F \in \FSet} \sum_n \mathds{1}_{\mathcal{I}_n}(F)\left\|\frac{\varphi_n(F)}{|F|}\right\|    \\
                    &   =       \sup_n \sup\left\{\left\|\frac{\varphi_n(F)}{|F|}\right\|: F \in \mathcal{I}_n\right\}   \\
                    &   \leq     \sup_n \sup\left\{\left\|\frac{\varphi_n(F)}{|F|}\right\|: F \textrm{ is $(K_n,\delta_n)$-invariant}\right\}   \\
                    &   \leq      \sup_n \vertG{\varphi_n} + 2^{-n}  \leq      C + 1 < \infty,
\end{align*}
Moreover, $\varphi$ is $G$-invariant, since for $F \in \mathcal{I}_n$ and $g \in G$, we have that $g \cdot F \in \mathcal{I}_n$ and
$$
g \cdot \varphi(F) = g \cdot \varphi_n(F) = \varphi_n(Fg^{-1}) = \varphi(g \cdot F),
$$
so $\varphi \in \Maps$. Note that $(\varphi_n)_n$ converges to $\varphi$. Indeed, given $n_0 \in \N$,
\begin{align*}
\vertG{\varphi - \varphi_{n_0}}   &   =       \inf_{(K,\delta)} \sup \left\{\sum_k\mathds{1}_{\mathcal{I}_k}(F)\left\|\frac{\varphi_k(F) - \varphi_{n_0}(F)}{|F|}\right\|: F \textrm{ is $(K,\delta)$-invariant}\right\}    \\
                    &   =       \inf_{n \geq n_0} \sup \left\{\sum_k \mathds{1}_{\mathcal{I}_k}(F)\left\|\frac{\varphi_k(F) - \varphi_{n_0}(F)}{|F|}\right\|: F \textrm{ is $(K_n,\delta_n)$-invariant}\right\}    \\
                    &   =       \inf_{n \geq n_0} \sup \left\{\sum_{k \geq n} \mathds{1}_{\mathcal{I}_k}(F)\left\|\frac{\varphi_k(F) - \varphi_{n_0}(F)}{|F|}\right\|: F \textrm{ is $(K_n,\delta_n)$-invariant}\right\}    \\
                    &   =        \inf_{n \geq n_0} \sup_{k \geq n} \sup \left\{\left\|\frac{\varphi_k(F) - \varphi_{n_0}(F)}{|F|}\right\|: F \in \mathcal{I}_k\right\}    \\
                    &   \leq        \inf_{n \geq n_0} \sup_{k \geq n_0} \sup \left\{\left\|\frac{\varphi_k(F) - \varphi_{n_0}(F)}{|F|}\right\|: F \in \mathcal{I}_k\right\}    \\
                    &   =       \sup_{k \geq n_0} \sup \left\{\left\|\frac{\varphi_k(F) - \varphi_{n_0}(F)}{|F|}\right\|: F \in \mathcal{I}_k\right\}    \\
                     &   \leq    \sup_{k \geq n_0} \sup \left\{\sum_{i = n_0}^{k-1} \left\|\frac{\varphi_{i+1}(F) - \varphi_i(F)}{|F|}\right\|: F \in \mathcal{I}_k\right\}    \\
                    &   \leq    \sup_{k \geq n_0} \sum_{i = n_0}^{k-1} \sup \left\{\left\|\frac{\varphi_{i+1}(F) - \varphi_i(F)}{|F|}\right\|: F \textrm{ is $(K_k,\delta_k)$-invariant}\right\}    \\
                    &   \leq    \sup_{k \geq n_0} \sum_{i = n_0}^{k-1} \sup \left\{\left\|\frac{\varphi_{i+1}(F) - \varphi_i(F)}{|F|}\right\|: F \textrm{ is $(K_i,\delta_i)$-invariant}\right\}    \\
                    &   \leq    \sup_{k \geq n_0} \sum_{i = n_0}^{k-1} 2^{-i+1}  = \inf_{n \geq n_0} \sum_{i = n_0}^{\infty} 2^{-i+1}  = \inf_{n \geq n_0} 2^{-n_0+2} = 2^{-n_0+2}. 
\end{align*}
Letting $n_0 \to \infty$, we conclude. Since $(\varphi_m)_m$ was an arbitrary $\vertG{\cdot}$-Cauchy sequence, the space $\Maps$ is $\vertG{\cdot}$-complete.
\end{proof}

Since $\NMaps$ is the kernel of $\vertG{\cdot}$, the semi-norm $\vertG{\cdot}$ induces a norm $\vertK{\cdot}$ on the quotient space $\Maps/\NMaps$ given by
$$
\vertK{\varphi + \NMaps} = \vertG{\varphi},
$$
where $\varphi$ is any representative of the class $\varphi + \NMaps$. A direct consequence of Proposition \ref{prop:cauchy} is the following.

\begin{corollary}
The space $(\Maps/\NMaps, \vertK{\cdot})$ is a Banach space.
\end{corollary}

\begin{remark}
Observe that, since $\vertG{\cdot} \leq \vertsup{\cdot}$, every subset $\mathcal{W} \subseteq \Maps$ that is $\vertG{\cdot}$-closed has to be $\vertsup{\cdot}$-closed. In particular, $\NMaps$ is a $\vertsup{\cdot}$-closed subspace of $\Maps$.
\end{remark}

\subsection{A characterization of the space of asymptotically additive set maps}

\begin{proposition}
\label{cor2}
The space $\AdMaps$ is $\vertG{\cdot}$-complete and the space $\AsMaps$ is $\vertG{\cdot}$-complete and $\vertG{\cdot}$-closed. Moreover, $\AsMaps=\AdMaps + \NMaps$ and $[\AdMaps]_\NMaps = [\AsMaps]_\NMaps$. In particular, $([\AdMaps]_\NMaps, \vertK{\cdot})$ and $([\AsMaps]_\NMaps, \vertK{\cdot})$ coincide and are Banach spaces.
\end{proposition}

\begin{proof}
By Theorem \ref{thm:cuneo}, $\varphi \in \AsMaps$ if and only if $\vertG{\varphi-\psi} = 0$ for some $\psi \in \AdMaps$. Since $\varphi = \psi + (\varphi-\psi)$ and $\varphi-\psi \in \NMaps$, we have that $\AsMaps = \AdMaps + \NMaps$ and $[\AsMaps]_\NMaps = [\AdMaps]_\NMaps$.

Let $(\psi_n)_n$ be a $\vertG{\cdot}$-Cauchy sequence in $\AdMaps$. Since $\Maps$ is $\vertG{\cdot}$-complete, there exists $\varphi \in \Maps$ such that $\vertG{\varphi - \psi_n} \to 0$. In particular, $\varphi$ is asymptotically additive and, by Theorem \ref{thm:cuneo}, there exists $\psi \in \AdMaps$ such that $\vertG{\varphi - \psi} = 0$. Then,
$$
\vertG{\psi - \psi_n} \leq \vertG{\psi - \varphi} + \vertG{\varphi - \psi_n} = \vertG{\varphi - \psi_n} \to 0,
$$
so $\AdMaps$ is $\vertG{\cdot}$-complete. Since $\AsMaps = \AdMaps + \NMaps$, it follows that $\AsMaps$ is $\vertG{\cdot}$-complete as well, by using additive representatives. Finally, to see that $\AsMaps$ is $\vertG{\cdot}$-closed, suppose that $\varphi \in \Maps$ and $(\varphi_n)_n$ is a sequence in $\AsMaps$ such that $\vertG{\varphi_n - \varphi} \to 0$. Since $\AsMaps$ is $\vertG{\cdot}$-complete and $\AsMaps = \AdMaps + \NMaps$, there exists $\psi \in \AdMaps$ such that $\vertG{\varphi_n - \psi} \to 0$. Therefore, $\vertG{\varphi - \psi} = 0$, so $\varphi = \psi + (\varphi-\psi) \in \AdMaps + \NMaps = \AsMaps$.
\end{proof}

\begin{proposition}
\label{cor:iso3}
If $U$ is a subspace of $V$ and $W \subseteq V$ is a subset such that $\Cob_W \subseteq U \subseteq \wCob$, then the subspaces $(\big[\overline{\langle W \rangle}\big], \|\cdot\|_U)$ and $(\big[\overline{\langle W \rangle}\big], \|\cdot\|_\wCob)$ are linearly isomorphic. If, in addition, $\pi$ is isometric, then both spaces are isometric. 
\end{proposition}

\begin{proof}
Consider the linear transformation $\theta: (\big[\overline{\langle W \rangle}\big], \|\cdot\|_U) \to (\big[\overline{\langle W \rangle}\big], \|\cdot\|_\wCob)$ given by $\theta(w + U) =  w + \wCob$ for $w \in \overline{\langle W \rangle}$. Since $U \subseteq \wCob$, the map $\theta$ is well-defined and surjective. By Corollary \ref{cor:LW},
$$
\|\theta(w+U)\|_\wCob = \|w+\wCob\|_\wCob \leq \limsup_{F \to G} \|A_F w\| \leq C_\pi \|w+U\|_U \leq \|w+\wCob\|_\wCob,
$$
so $\theta$ is also continuous and injective, and therefore a linear isomorphism. If, in addition, $\pi$ is isometric, then $C_\pi = 1$, so $\|\theta(w+U)\|_\wCob = \|w+U\|_U$ and $\theta$ is an isometry.
\end{proof}

If $U$ is a closed subspace of $V$, we know that the quotient space $(V/U,\|\cdot\|_U)$ is a Banach space. Considering this, we have the following.

\begin{theorem}
\label{thm:iso4}

Let $U$ be a subspace of $V$ such that $U \subseteq \wCob$. Consider the map $\hat{S}: V/U \to [\AdMaps]_\NMaps$ given by $\hat{S}(v + U) = S(v) + \NMaps$.  Then the map $\hat{S}$ is a surjective and continuous linear transformation from $(V/U,\|\cdot\|_U)$ onto $([\AdMaps]_\NMaps, \vertK{\cdot})$. Moreover, if $U = \wCob$, then $\hat{S}$ is a linear isomorphism and if, additionally, $\pi$ is isometric, then $\hat{S}$ is an isometry.
\end{theorem}

\begin{proof}
First, note that $S(v) + \NMaps \in [\AdMaps]_\NMaps$ for any $v \in V$. Next, note that the map $\hat{S}$ does not depend on the representative. Indeed, if we pick $v_1,v_2 \in v + U$, then $v_1 - v_2 \in U$, and by Proposition \ref{prop:cob1},
\begin{align*}
\vertK{(S(v_1) + \NMaps) - (S(v_2)+\NMaps)} &   =   \vertK{S(v_1-v_2) + \NMaps} \\
                                            &   \leq \vertG{S(v_1-v_2)}    \\
                                            &   \leq C_\pi \left\|(v_1-v_2) + U\right\|_U \leq C_\pi \left\|U\right\|_U = 0.  
\end{align*}

We now verify that $\hat{S}$ is linear and continuous. Given $v_1 + U, v_2 + U \in V/U$ and $\alpha \in \R$, we have that
\begin{align*}
\hat{S}((v_1 + U) + \alpha(v_2 + U))    &   =   \hat{S}(v_1 + \alpha v_2 + U)   \\
                                                &   =   S(v_1 + \alpha v_2) + \NMaps \\
                                                &   =   (S(v_1) + \NMaps) + (\alpha S(v_2) + \NMaps)    \\
                                                &   = \hat{S}(v_1 + U) + \alpha \hat{S}(v_2 + U),
\end{align*}
so $\hat{S}$ is linear. In addition, by Proposition \ref{prop:cob1} again,
\begin{align*}
\vertK{\hat{S}(v + U)} = \vertK{S(v) + \NMaps} = \vertG{S(v)} \leq C_\pi\|v + U\|_U,
\end{align*}
so $\hat{S}$ is continuous. To see that $\hat{S}$ is surjective, consider an arbitrary $\psi + \NMaps$ in $[\AdMaps]_\NMaps$ with $\psi \in \AdMaps$. Then, since $\psi$ is additive,
$$
\hat{S}(\psi(\{1_G\}) + U) = S(\psi(\{1_G\})) + \NMaps = \psi + \NMaps.
$$
This proves that $\hat{S}$ is a surjective and continuous linear transformation.

Assume that $U = \wCob$ and let us prove that $\hat{S}$ is injective. If $\hat{S}(v_1+\wCob) = \hat{S}(v_2+\wCob)$, then $S(v_1) + \NMaps = S(v_2) + \NMaps$, so $S(v_1-v_2) \in \NMaps$ and
$$
0 = \vertG{S(v_1-v_2)} \geq \|(v_1-v_2) + \Cob\|_{\Cob} \geq \|(v_1-v_2) + \wCob\|_{\wCob},
$$
where we have appealed to Proposition \ref{prop:cob2}. Therefore, $(v_1-v_2) + \wCob = \wCob$, i.e., $v_1+\wCob = v_2+\wCob$.

Finally, since $V/\wCob$ and $[\AdMaps]_\NMaps$ are Banach spaces, by the Bounded Inverse Theorem, the inverse $\hat{S}^{-1}: [\AdMaps] \to V/\wCob$ given by $\hat{S}^{-1}(\psi + \NMaps) = \psi(\{1_G\})$ is also continuous. Then, $\hat{S}$ is a linear isomorphism between $V/\wCob$ and $[\AdMaps]_\NMaps$. Moreover, if $\pi$ is isometric (i.e., if $C_\pi = 1$), $\hat{S}$ is an isometry, since
$$
\vertK{\hat{S}(v+\wCob)} = \vertK{S(v)+\NMaps} = \vertG{S(v)} = \|v + \wCob\|_\wCob.
$$
\end{proof}

The diagram in Figure \ref{diagram} summarizes all the results of this section. Here, $U$ is a subspace of $V$ and $W \subseteq V$ such that $\Cob_W \subseteq U \subseteq \wCob$, and the map $j$ is defined by $j(\psi) = \psi$, which is continuous since $\vertG{\cdot} \leq \vertsup{\cdot}$.

\begin{figure}[ht]
\label{diagram}
\centering
\begin{tikzpicture}[baseline]
    \node[draw, rounded corners, inner xsep=0.1cm, inner ysep=0.1cm](box){
$$
\begin{tikzcd}[row sep=1.0cm, column sep=0.6cm]
                &   (\wCob,\|\cdot\|) \arrow[d, hook] \arrow[r,"\sim","S"'] &  (\AdNMaps,\vertsup{\cdot}) \arrow[d, hook] &   \\
                &   (V,\|\cdot\|) \arrow[r, "\sim","S"'] \arrow[dl, "+U"',twoheadrightarrow] \arrow[d, "+\wCob",twoheadrightarrow] &   (\AdMaps,\vertsup{\cdot}) \arrow[d, "+\AdNMaps",twoheadrightarrow] \arrow[r, "j"]  &  (\AsMaps,\vertG{\cdot}) \arrow[dd, "+\NMaps"] \\
(V/U,\|\cdot\|_{U}) \arrow[r, "+\wCob",twoheadrightarrow]    & (V/\wCob,\|\cdot\|_{\wCob}) \arrow[rd,"\sim", "\hat{S}"'] \arrow[r, "\sim","\overline{S}"'] &   (\AdMaps/\AdNMaps,\vertiii{\cdot}_{\AdNMaps})  \arrow[d, "\sim"]  &  \\
(\big[\overline{\langle W \rangle}\big],\|\cdot\|_{U}) \arrow[r, "\sim","+\wCob"'] \arrow[u, hook]    & (\big[\overline{\langle W \rangle}\big],\|\cdot\|_{\wCob}) \arrow[u,hook]  & ([\AdMaps],\vertK{\cdot}) \ar[equal]{r}  &  ([\AsMaps],\vertK{\cdot})
\end{tikzcd}
$$
};
\end{tikzpicture}
\caption{Summary of relations between spaces.}
\end{figure}

\section{Relative asymptotical additivity}
\label{sec:relative}

A natural, and relevant for applications, question that stems from the result by Cuneo \cite[Theorem 1.2]{cuneo2020additive} is whether asymptotically additive sequences of functions with certain regularity have an additive realization with the same regularity. In the applications given in \cite{cuneo2020additive} to thermodynamic formalism this is a fundamental question that remained open. Of course, the same can be asked in relation to Theorem \ref{thm:cuneo}. In this section we provide a definition of asymptotically additive with respect to a subset and go on to establish conditions so that additive realizations belong to it.

Given $\varphi \in \Maps$ and a subset $W \subseteq V$, we say that $\varphi$ is {\bf asymptotically additive relative to $W$} if
$$
\inf_{w \in W} \limsup_{F \to G} \left\|\frac{\varphi(F)}{|F|}-A_F w\right\| = 0.
$$
Clearly, being asymptotically additive relative to $W$ is a stronger property than merely being asymptotically additive, which corresponds to the particular case $W = V$. In particular, by Corollary \ref{cor:cuneo}, there exists an additive realization $v \in V$ such that
$$
\limsup_{F \to G} \left\|\frac{\varphi(F)}{|F|}-A_F v\right\| = 0.
$$
We are interested in studying when it is possible to chose an additive realization $v$ in $W$.

\subsection{Geometrical considerations}

Denote by $\AsMaps(W)$ the set of maps in $\Maps$ that are asymptotically additive relative to $W$.

\begin{lemma}
\label{lem:closed}
The set $\AsMaps(W)$ is $\vertG{\cdot}$-closed.
\end{lemma}

\begin{proof}
Let $(\varphi_n)_n$ be a sequence in $\AsMaps(W)$ and $\varphi \in \Maps$ such that $\vertG{\varphi_n - \varphi} \to 0$ as $n \to \infty$. It suffices to prove that $\varphi$ belongs to $\AsMaps(W)$. Indeed, given $w \in W$ and $n \in \N$,
\begin{align*}
\limsup_{F \to G} \left\|\frac{\varphi(F)}{|F|} - A_F w\right\| &    \leq   \limsup_{F \to G} \left(\left\|\frac{\varphi(F)}{|F|} - \frac{\varphi_n(F)}{|F|}\right\| + \left\|\frac{\varphi_n(F)}{|F|} - A_F w\right\|\right)   \\
                    &   \leq    \vertG{\varphi-\varphi_n} + \limsup_{F \to G} \left\|\frac{\varphi_n(F)}{|F|} - A_F w\right\|.
\end{align*}
Taking the infimum over $w \in W$, we obtain that
$$
\inf_{w \in W}\limsup_{F \to G} \left\|\frac{\varphi(F)}{|F|} - A_F w\right\|    \leq    \vertG{\varphi-\varphi_n} + \inf_{w \in W}\limsup_{F \to G} \left\|\frac{\varphi_n(F)}{|F|} - A_F w\right\| = \vertG{\varphi-\varphi_n},
$$
and letting $n \to \infty$, the result follows.
\end{proof}

\begin{lemma}
\label{lem:convex}
If $W$ is convex, then $\AsMaps(W)$ is convex.
\end{lemma}

\begin{proof}
Given maps $\varphi_1,\varphi_2 \in \AsMaps(W)$ and $n \in \N$, let $w_1,w_2 \in W$ such that
$$
\limsup_{F \to G} \left\|\frac{\varphi_i(F)}{|F|}-A_F w_i\right\| \leq 2^{-n} \quad \text{for $i=1,2$.}
$$
Since $W$ is convex, we have that $\lambda w_1+(1-\lambda)w_2 \in W$ for $\lambda \in [0,1]$. Then,
\begin{align*}
        &   \limsup_{F \to G} \left\|\frac{(\lambda\varphi_1+(1-\lambda)\varphi_2)(F)}{|F|}-A_F(\lambda w_1+(1-\lambda)w_2)\right\|  \\
\leq    & \quad   \lambda\limsup_{F \to G} \left\|\frac{\varphi_1(F)}{|F|}-A_Fw_1\right\| + (1-\lambda)\limsup_{F \to G} \left\|\frac{\varphi_2(F)}{|F|}-A_Fw_2\right\|   \\
\leq    &  \quad \lambda 2^{-n} + (1-\lambda)2^{-n} = 2^{-n}.
\end{align*}

Since $n$ was arbitrary, we conclude that
$$
\inf_{w \in W}\limsup_{F \to G} \left\|\frac{(\lambda\varphi_1+(1-\lambda)\varphi_2)(F)}{|F|}-A_Fw\right\| = 0.
$$
\end{proof}

\subsection{A classification theorem for additive realizations}
In our next result we characterize the set to which  additive realizations of asymptotically additive relative to a subset $W$ belong. Interestingly, we decompose this set into two parts and prove that all additive realizations belong to the same one.

\theoremtwo*

\begin{proof}
Notice that $\varphi: \FSet \to V$ is asymptotically additive relative to $W$ if and only if, for each $n \in \N$, there exists $w_n \in W$ such that $\vertG{\varphi - S(w_n)} \leq 2^{-n}$, so
$$
\vertG{S(w_n) - S(w_m)} \leq 2^{-n} + 2^{-m} \quad \text{for all}~n,m \in \N,
$$
hence $(S(w_n))_n$ is a Cauchy sequence in $(\AsMaps,\vertG{\cdot})$. Next, observe that $(S(w_n))_n$ is a Cauchy sequence in $(\AsMaps,\vertG{\cdot})$ if and only if $(S(w_n) + \NMaps)_n$ is a Cauchy sequence in $([\AsMaps],\vertK{\cdot})$, because $\vertK{\varphi' + \NMaps} = \vertG{\varphi'}$ for every $\varphi' \in \AsMaps$. By Proposition \ref{cor2}, $[\AsMaps] = [\AdMaps]$ and, equivalently, $(S(w_n) + \NMaps)_n$ is a Cauchy sequence in $([\AdMaps],\vertK{\cdot})$. By Theorem \ref{thm:iso4}, this is equivalent to $(w_n + \wCob)_n$ being a Cauchy sequence in $(V/\wCob,\|\cdot\|_\wCob)$. Notice that, since $w_n \in W$ for each $n$, $(w_n + \wCob)_n$ is a Cauchy sequence in $(\big[\overline{\langle W \rangle}\big], \|\cdot\|_\wCob)$ as well. By Proposition \ref{cor:iso3}, $(w_n + \Cob_W)_n$ is a Cauchy sequence in $(V/\Cob_W,\|\cdot\|_{\Cob_W})$. Since $(V/\Cob_W,\|\cdot\|_{\Cob_W})$ is complete, there exists $v \in V$ such that $(w_n + \Cob_W)_n$ converges to $v + \Cob_W$ in $V/\Cob_W$, which implies that $v \in \overline{W + \Cob_W}$. It then follows that $v$ is an additive realization of $\varphi$, because $\lim_n (w_n + \Cob_W) = v + \Cob_W$ implies $\lim_n (w_n + \wCob) = v + \wCob$ and
$$
\vertG{\varphi - S(v)} \leq \vertG{\varphi - S(w_n)} + \vertG{S(w_n)-S(v)} \leq 2^{-n} + \vertG{S(w_n)-S(v)} \to 0
$$
as $n \to \infty$. This proves that $(A1) \implies (A2)$. To see the converse, observe that if $v \in \overline{W + \Cob_W}$ is an additive realization of $\varphi$, then there exists a sequence $(w_n)_n$ in $W$ such that $w_n + \Cob_W$ converges to $v + \Cob_W$, so
$$
\vertG{\varphi - S(w_n)} \leq \vertG{\varphi - S(v)} + \vertG{S(v) - S(w_n)} = \vertG{S(v) - S(w_n)},
$$
and 
$$
\inf_{w \in W} \vertG{\varphi - S(w)} \leq \inf_{n \in \N} \vertG{\varphi - S(w_n)} = \inf_{n \in \N} \vertG{S(v) - S(w_n)} = 0.
$$

By (A2), we know that there exists an additive realization $v \in \overline{W+\Cob_W}$ of $\varphi$ and, by Lemma \ref{lem:witness}, we know that $\witness = v + \wCob$. Then, $\witness \subseteq \overline{W+\Cob_W} + \wCob \subseteq \overline{W+\Cob}$, since $\Cob_W \subseteq \Cob$. Therefore, $\witness \cap (W + \wCob) = \emptyset$ or $\witness \cap (W + \wCob) \neq \emptyset$. In the former, we conclude directly that (B2) holds. In the latter, we conclude that $\witness = (w + \wCob) + \wCob = w + \wCob$ for some $w \in W$, so (B1) holds.
\end{proof}

\begin{remark}
\label{rem:tool}
Every $v \in \overline{W}$ is an additive realization for some set map $\varphi: \FSet \to \langle W \rangle$ asymptotically additive relative to $W$. It suffices to consider a sequence $(v_n)_n$ in $W$ that converges to $v$ and let $\varphi(F) = S_F v_{|F|}$ for every $F \in \FSet$. Since $S_F v_{|F|}$ is just a linear combination of elements in $W$, we have that $\varphi(F)$ belongs to $\langle W \rangle $ for every $F \in \FSet$. In addition, $v$ is an additive realization of $\varphi$, since
$$
\lim_{F \to G} \left\|\frac{\varphi(F)}{|F|}-A_F v\right\|  =   \lim_{F \to G} \left\|A_F(v_{|F|}-v)\right\|  \leq   \lim_{F \to G} C_\pi\left\|v_{|F|}-v\right\| = 0,
$$
where we have used that $\pi$ is uniformly bounded. Notice, however, that this fact does not address the question whether there are examples in which (B1) and (B2) from Theorem \ref{thm:cuneo-rel} actually occur. Such examples will be provided in \S \ref{subsec:examples} and \S \ref{subsec:examples-thermo}.
\end{remark}

\begin{corollary}
\label{cor:cuneo2}
Let $G$ be a countable amenable group, $(V,\|\cdot\|)$ a complete semi-normed space, and $\pi: G \to \mathrm{Isom}(V)$ a uniformly bounded representation. Given a non-empty subset $W \subseteq V$ and a subspace $U$ such that $\Cob_W \subseteq U \subseteq \wCob$, the following are equivalent:
\begin{enumerate}
    \item $W + U$ is closed in $V$.
    \item For every set map $\varphi: \FSet \to V$ that is asymptotically additive relative to $W$, there exists an additive realization of $\varphi$ in $W$.
\end{enumerate}
\end{corollary}

\begin{proof}
If $W + U$ is closed in $V$, then $\overline{W+\Cob_W} \subseteq W + U \subseteq W + \wCob$. Therefore, if $\varphi: \FSet \to V$ is asymptotically additive relative to $W$, there exists an additive realization $v \in W + \wCob$, i.e., $\witness = w + \wCob$ for some $w \in W$; in particular, there exists an additive realization, namely $w$, of $\varphi$ in $W$.

Let us suppose that $\varphi: \FSet \to V$ is asymptotically additive relative to $W$ and there is no additive realization $w \in W$ of $\varphi$. Then, we know that there exists an additive realization in $\overline{W + \Cob_W}$ and there is no additive realization in $W + \wCob$. In particular,
$$
\overline{W + U} \setminus (W + U) \supseteq \overline{W + \Cob_W} \setminus (W + \wCob) \neq \emptyset,
$$
so $W + U$ is not closed in $V$.
\end{proof}

\begin{remark}
As discussed in \S \ref{sec:folner}, we can also consider the notion of asymptotical additivity relative to a subset $W$ by restricting ourselves to a particular F{\o}lner sequence $(F_n)_n$, i.e.,
$$
\inf_{w \in W} \limsup_{n \to \infty} \left\|\frac{\varphi(F_n)}{|F_n|}-A_{F_n} w\right\| = 0.
$$
In this case, it is easy to see that everything can be adapted accordingly, by replacing
$$
\lim_{F \to G} \left\|\frac{\varphi(F)}{|F|}-A_{F} w\right\| = 0.
$$
with the weaker conclusion
$$
\lim_{n \to \infty} \left\|\frac{\varphi(F_n)}{|F_n|}-A_{F_n} w\right\| = 0.
$$
\end{remark}

\begin{remark}
The condition asking for $W + U$ to be closed is satisfied in natural scenarios. For example, if $W$ is a closed and $G$-invariant subspace of $V$ (i.e., $\pi(G)W = W$), then $\Cob_W \subseteq W$ and $W + \Cob_W = W$, so $W + \Cob_W$ is closed. 
\end{remark}

\subsection{Examples}
\label{subsec:examples}
We provide several examples exhibiting the different behaviors that occur in Theorem \ref{thm:cuneo-rel}.

\begin{example}[$L^p$ spaces]
Let $G \acts (X,\mu)$ be an ergodic probability measure preserving action. Consider the Banach space $V = L^1_\mu(X)$ endowed with the norm $\|\cdot\| = \|\cdot\|_1$ and the associated Koopman representation $\pi: G \to \Isom(V)$ given by $\pi(g)f(x) = f(g^{-1}\cdot x)$ for $x \in X$ $\mu$-a.s. Given $p \geq 1$, take the subset $W = L^p_\mu(X) \subseteq V$. Since $L^p_\mu(X)$ is $\|\cdot\|_1$-dense in $L^1_\mu(X)$, we have that $\wCob_{W} = \wCob$. In addition, since the system is ergodic, the only $G$-invariant functions are the constants ones. Then (e.g., see \cite[Theorem 4.28]{kerr2016ergodic}), if we take $U = \wCob_W$,
$$
W + U = L^p_\mu(X) + \wCob_{L^p_\mu(X)} = \R\mathds{1} + \wCob_{L^1_\mu(X)} = L^1_\mu(X) = V,
$$
where $\R\mathds{1}$ denotes the subspace of real-valued constant functions. Therefore, $W + U$ is closed, and Corollary \ref{cor:cuneo2} assures that for every set map $\varphi: \FSet \to L^1_\mu(X)$ that is asymptotically additive relative to $L^p_\mu(X)$, there exists an additive realization of $\varphi$ in $L^p_\mu(X)$.

Note that our results remain valid if we replace the Banach spaces of classes of functions $L^p_\mu(X)$ by the corresponding complete semi-normed spaces of functions $\mathcal{L}^p_\mu(X)$.
\end{example}

\begin{example}[Bounded and continuous functions]
Fix a compact metric space $X$ and let $G \acts (X,\mu)$ be a continuous and ergodic probability measure preserving action. Consider the Banach space $V =\mathcal{B}(X)$ of bounded real-valued functions endowed with the uniform norm $\|\cdot\| = \|\cdot\|_\infty$ and the associated Koopman representation $\pi: G \to \Isom(V)$ given by $\pi(g)f(x) = f(g^{-1}\cdot x)$ for every $x \in X$. Let $E \subseteq \R$ be a closed set  and consider the subset $W = \{f \in \mathcal{C}(X): \int{f}d\mu \in E\} \subseteq V$. Then, since $\mathcal{C}(X)$ is closed in $\mathcal{B}(X)$ and $E$ is closed in $\R$, it follows that $W$ is closed in $V$. Moreover, since $\mu$ is $G$-invariant, the set $W$ is $G$-invariant, so $\Cob_W \subseteq W$. Therefore, $W + \Cob_W = W$ is closed, and for every set map $\varphi: \FSet \to \mathcal{B}(X)$ that is asymptotically additive relative to $W$, there exists an additive realization $h \in W$ of $\varphi$.
\end{example}

\begin{example}[Walters condition]
Let $X \subseteq \{0,1\}^\Z$ be a transitive subshift and  $T: X \to X$ the shift map. Denote  by $\mathrm{Wal}(X,T)$ the set of functions $f: X \to \R$ that satisfies the \emph{Walters condition}, that is, for every $\epsilon>0$, there exists $\delta>0$ such that
$$
d_n(x, y) \leq \delta \implies \left|\sum_{i=0}^{n-1} f(T^ix)-\sum_{i=0}^{n-1} f(T^iy)\right| \leq \epsilon,
$$
for all $n \in \N$ and for all $x, y \in X$, where $d_n$ denotes the \emph{$n$-th Bowen metric} (see \cite{bousch2001condition}).

Let $V = \mathcal{C}(X)$ and let $\pi: \Z \to \Isom(V)$ be the associated Koopman representation given by $\pi(-n)f(x) = f(T^n x)$ for every $x \in X$ and $n \in \Z$. Consider $W = \mathrm{Wal}(X,T)$. It is well-known that $W$ is dense but not equal to $V$ \cite{bousch2001condition}. We will establish that there exists a set map $\varphi: \FSet \to V$ that is asymptotically additive relative to $W$ for which there is no additive realization for $\varphi$ in $W$. To do this, due to Corollary \ref{cor:cuneo2}, it suffices to verify that $W + U$ is not closed for some subspace $\Cob_W \subseteq U \subseteq \wCob$. In this case, it is convenient to consider $U = \Cob$, since by \cite[Proposition 1]{bousch2001condition}, every coboundary satisfies the Walters condition, i.e., $\Cob \subseteq W$. Therefore, $W + U = W$, that is not closed in $V$.
\end{example}

\begin{example}
Fix a compact metric space $X$ and let $G \acts X$ be a continuous action. Consider the dual space $V =\mathcal{C}(X)^*$ endowed with the linear operator norm and the representation $\pi: G \to \Isom(V)$ given by $\pi(g) L(f) = L(\kappa(g)f)$ for every $L \in \mathcal{C}(X)^*$, $f \in \mathcal{C}(X)$, and $g \in G$, where $\kappa: G \to \Isom(\mathcal{C}(X))$ is the associated Koopman representation.

Take $W$ to be the set $\mathcal{P}(X)$ of Borel probability measures on $X$, through the identification given by the Riesz representation theorem. Then, $W$ is closed (as it is closed for the weak* topology) and $G$-invariant. Therefore, Corollary \ref{cor:cuneo2} assures that for every set map $\varphi: \FSet \to \mathcal{C}(X)^*$ that is asymptotically additive relative to $\mathcal{P}(X)$, there exists an additive realization of $\varphi$ in $\mathcal{P}(X)$. Similarly, if $W$ is the set $\mathcal{P}_G(X)$ of $G$-invariant Borel probability measures on $X$, then $W$ is closed and $G$-invariant. Moreover, its points are fixed by $\pi$ and $\Cob_W = \{0\}$, and the same conclusion holds.
\end{example}

\section{Applications to thermodynamic formalism} \label{sec:thermo}

A natural and relevant application of our results is in the realm of thermodynamic formalism.
Indeed, as observed by Cuneo \cite{cuneo2020additive} for $\Z$-actions, the thermodynamic formalism developed to study asymptotically additive sequences \cite{feng-huang} can be reduced 
 to the classical case of continuous functions by means of additive realizations. Following the same line of thought, in this section  we extend the thermodynamic formalism so as to consider general group actions and the related asymptotically additive maps. By means of Theorem \ref{thm:cuneo} and Corollary \ref{cor:cuneo}, we are able to reduce it to the classical case.  We define the pressure and the notion of equilibrium state and go on to prove the variational principle. Our results extend the theory  for group actions, as studied in  \cite{Moulin,stepin}, to the asymptotically additive setting.
We stress that we are able to deal with both compact and non-compact phase spaces. Moreover, we obtain new results even in the classical setting of $\Z$-actions.

An open question that arose from Cuneo's work was whether asymptotically additive sequences of functions with certain regularity have an additive realization with the same regularity. In order to understand the relevance of the question, recall that the thermodynamic formalism of regular functions is very well behaved. In the $\Z$-action setting when the phase space is compact (e.g., see \cite{bowen}), there exists a unique equilibrium state, the pressure function is real analytic, and the equilibrium state has strong ergodic properties. In order to address this question, Theorem \ref{thm:cuneo-rel} is a fundamental result. It establishes a dichotomy, either an additive realization belongs to the same regularity class $W$ or there exists an additive realization in the complement of the regularity class modulo weak coboundaries, that is, in $V \setminus (W + \wCob)$. In Example \ref{ex:hol}, we show that both instances of this dichotomy can occur for natural regularity classes. Theorem \ref{thm:cuneo-rel} clarifies what has been observed in the setting of thermodynamic formalism of matrix cocycles \cite[Theorem 2.9]{bkm}. Indeed, in \cite{bkm} an asymptotically additive sequence of H\"older functions is constructed for which the additive realization is not H\"older.
Similar examples were also obtained in \cite{holanda}.

\subsection{Pressure} \label{sec_pres} 

The mathematical notion of pressure was brought from statistical mechanics into dynamics in the early 1970s by Ruelle and Sinai among others in the context of $\Z$- and $\Z^d$-actions over compact spaces \cite{ruelle,sinai}. This definition was subsequently extended to actions of amenable groups; see for example \cite{ollagnier2007ergodic,Moulin,bufetov2011pressure}. The compactness assumption was dropped for symbolic $\Z$-actions in \cite{mauldin1996dimensions, sarig1999thermodynamic} and for actions of $\Z^d$ and general amenable groups in \cite{muir2011new} and \cite{beltran2024thermodynamic}, respectively. Non-additive versions of thermodynamic formalism in which the pressure is evaluated in a sequence of functions instead of a single function, were introduced by Falconer in the early 1980s \cite{falconer1988}. This theory has been greatly expanded so as to deal with different types of cocycles (e.g., see \cite{barreira1,liang2012topological,yan2013sub}). In this section we define the pressure for amenable group actions and asymptotically additive maps in the symbolic setting.

Consider the {\bf $G$-full shift} over $\N$, that is, the set $\N^G = \{x: G \to \N\}$ of $\N$-colorings of $G$, endowed with the prodiscrete topology, where $\N$ denotes the set of non-negative integers. Fix a strictly increasing sequence $(E_m)_m$ in $\FSet$ such that $E_1 = \{1_G\}$ and $\bigcup_m E_m=G$. Then, the topology of $\N^G$ is metrizable by the metric $d$ given by $d(x, y)=2^{-\inf \left\{m \geq 1: x_{E_m} \neq y_{E_m}\right\}}$, where $x_F$ denotes the restriction of $x$ to a set $F \subseteq G$.

The group $G$ acts on $\N^G$ by translations, i.e.,
\begin{equation*}
(g \cdot x)(h) = x(hg) \quad \text{for every}~x \in X~\text{and}~g,h \in G.
\end{equation*}
A {\bf $G$-subshift} will be any closed and $G$-invariant subset $X \subseteq \N^G$. We will denote by $X_F=\left\{x_F: x \in X\right\}$ the set of restrictions of $x \in X$ to $F$. The sets of the form $[w]=\left\{x \in X: x_F=w\right\}$, for $w \in X_F$ and $F \in \FSet$, are called cylinder sets. The family of such sets is the standard basis for the subspace topology of $X$. It is well-known that $X = A^G$ is a compact $G$-subshift for any finite subset $A \subseteq \N$.

Fix a $G$-subshift $X$. Consider the linear space of real-valued continuous functions $\mathcal{C}(X)$. In this context, the Koopman representation $\pi:G \to \mathrm{Isom}(\mathcal{C}(X))$ is given by
$$
\pi(g)\phi(x) = \phi(g^{-1} \cdot x) \quad \text{for every}~g \in G, \phi \in \mathcal{C}(X),~\text{and}~x \in X.
$$

Given a set map $\varphi: \FSet \to \mathcal{C}(X)$, we define the {\bf partition function for $\varphi$ on $F$} as
\begin{equation*}
\hat{Z}_F(\varphi):= \sum_{w \in X_F} \exp \left(\sup \varphi(F)([w])\right),
\end{equation*}
where $\sup\varphi(F)([w]) = \sup\{\varphi(F)(x): x \in [w]\}$. We define the {\bf pressure of $\varphi$}, which we denote by $\hat{P}(\varphi)$, as 
$$
\hat{P}(\varphi) := \lim_{F \to G} \frac{1}{|F|} \log \hat{Z}_F(\varphi),
$$
whenever the limit exists.

Consider again the map $S: \mathcal{C}(X) \to \AdMaps$ given by
$$
S(\phi)(F) = S_F(\phi) = \sum_{g \in F} \phi(g \cdot ~).
$$
Let $\phi \in \mathcal{C}(X)$ and $F \in \FSet$, and define the {\bf partition function for $\phi$ on $F$} as
\begin{equation*}
Z_F(\phi):= \hat{Z}_F(S(\phi)) = \sum_{w \in X_F} \exp \left(\sup S_F(\phi)([w])\right).
\end{equation*}
We define the {\bf pressure of $\phi$}, which we denote by $P(\phi)$, as 
$$
P(\phi) := \hat{P}(S(\phi)) = \lim_{F \to G} \frac{1}{|F|} \log Z_F(\phi),
$$
whenever the limit exists. Note that the definitions of $Z_F(\phi)$ and $P(\phi)$ recover the standard ones.

We will say that $\phi \in \mathcal{C}(X)$ has {\bf finite oscillation} if
$$
\delta(\phi) := \sup \left\{|\phi(x)-\phi(y)|: x(1_G)=y(1_G)\right\}<\infty,
$$
and that is {\bf exp-summable} if
$$
\sum_{a \in \N} \exp (\sup\{\phi(x): x(1_G) = a\})<\infty.
$$

\begin{lemma}
\label{lem:pressure}
Let $X$ be a $G$-subshift. If $\varphi \in \Maps$ and $\phi \in \mathcal{C}(X)$ are such that
$$
\vertG{\varphi - S(\phi)} = 0,
$$
then $\hat{P}(\varphi)$ exists if and only if $P(\phi)$ exists. Moreover, if $P(\phi)$ exists, then $\hat{P}(\varphi) = P(\phi)$. 
\end{lemma}

\begin{proof}
Due to the assumption on $\varphi$ and $\phi$, we have
$$
\lim_{F \to G} \left\|\frac{\varphi(F)}{|F|} - A_F(\phi) \right\|_\infty = \vertG{\varphi - S(\phi)} = 0.
$$

Fix $\epsilon > 0$ and consider $(K,\delta) \in \FSet \times \R_{\geq 0}$ such that
$$
\left\|\varphi(F) - S_F(\phi) \right\|_\infty < \epsilon|F|
$$
for every $(K,\delta)$-invariant set $F$. Then, for every $x \in [w]$,
$$
S_F(\phi)(x) - \epsilon|F| \leq \varphi(F)(x) \leq S_F(\phi)(x) + \epsilon|F|,
$$
so
$$
\sup S_F(\phi)([w]) - \epsilon|F| \leq \sup \varphi(F)([w]) \leq \sup S_F(\phi)([w]) + \epsilon|F|
$$
and
$$
Z_F(\phi) \exp(-\epsilon|F|) \leq \hat{Z}_F(\varphi) \leq Z_F(\phi) \exp(\epsilon|F|).
$$

After taking $\log(\cdot)$ and dividing by $|F|$, we obtain that
$$
- \epsilon \leq \frac{\log \hat{Z}_F(\varphi)}{|F|} - \frac{\log Z_F(\phi)}{|F|} \leq \epsilon,
$$
and since $\lim_{F \to G}$ is additive, the result follows.
\end{proof}

If $X$ is compact, it is known that $P(\phi)$ always exists \cite{ollagnier2007ergodic}. In \cite{beltran2024thermodynamic}, it was shown that when $X = \N^G$, if $\phi$ is exp-summable and uniformly continuous with finite oscillation, then $P(\phi)$ exists as well. Considering this, we have the following result.

\begin{proposition}
\label{prop:varprin}
Let $X$ be a $G$-subshift. If $\varphi \in \Maps$ is asymptotically additive, then there exists $\phi \in \mathcal{C}(X)$ such that
$$
\vertG{\varphi - S(\phi)} = 0.
$$
If $X$ is compact, then $\hat{P}(\varphi)$ exists and $\hat{P}(\varphi) = P(\phi)$, and if $X = \N^G$ and $\phi$ is exp-summable and uniformly continuous potential with finite oscillation, the same conclusion holds. 
\end{proposition}

\begin{proof}
By Corollary \ref{cor:cuneo}, there exists $\phi \in \mathcal{C}(X)$ such that $\vertG{\varphi - S(\phi)} = 0$. If $X$ is compact or $X = \N^G$ and $\phi$ is exp-summable and uniformly continuous potential with finite oscillation, then $P(\phi) = \lim_{F \to G} \frac{1}{|F|} \log Z_F(\phi)$ exists. By Lemma \ref{lem:pressure}, we conclude that $\hat{P}(\varphi)$ exists and $\hat{P}(\varphi) = P(\phi)$.
\end{proof}

\subsection{Variational principle and equilibrium states}

A fundamental result in thermodynamic formalism that relates the (topological) pressure with the (measure-theoretic) free energy is called \emph{variational principle}. It was first proved, in complete generality, by Walters \cite{walters-var} for $\Z$-actions in 1976. Ever since, generalizations for other group actions (e.g., \cite{paulo, Moulin, stepin}) and for non-additive versions of the pressure (e.g., \cite{cao, yan2013sub}) have been obtained. In this section, we prove that the pressure defined in \S \ref{sec_pres} satisfies the corresponding variational principle.

Recall that $\mathcal{P}_G(X)$ denotes the set of $G$-invariant Borel probability measures on $X$.

\begin{lemma}
\label{lem:varprin1}
Let $X$ be a $G$-subshift. If $\varphi \in \Maps$ and $\phi \in \mathcal{C}(X)$ are such that $$\vertG{\varphi - S(\phi)} = 0,$$
then
$$
\lim_{F \to G} \int{\frac{\varphi(F)}{|F|}}d\nu = \int{\phi}d\nu \quad \text{for all}~\nu \in \mathcal{P}_G(X).
$$

\end{lemma}

\begin{proof}
Take $\nu \in \mathcal{P}_G(X)$ and fix $\epsilon > 0$. Consider $(K,\delta) \in \FSet \times \R_{\geq 0}$ such that
$$
\left\|\varphi(F) - S_F(\phi) \right\| < \epsilon|F|
$$
for every $(K,\delta)$-invariant set $F$. Therefore,
\begin{align*}
\left|\int{\varphi(F)}d\nu - |F|\int{\phi}d\nu \right|  &   = \left|\int{(\varphi(F) - S_F(\phi))}d\nu \right|  \\
                                                        &   \leq \int{\left\|\varphi(F) - S_F(\phi) \right\|}d\nu   \\
                                                        &   < \int{\epsilon|F|}d\nu = \epsilon|F|.
\end{align*}

Dividing by $|F|$, we obtain that
$$
\left|\int{\frac{\varphi(F)}{|F|}}d\nu - \int{\phi}d\nu \right|  <  \epsilon
$$
for every $(K,\delta)$-invariant set $F$. Since for every $\epsilon > 0$ we can find $(K,\delta)$ such that for every $(K,\delta)$-invariant set $F$ the equation above is satisfied, we conclude that
$$
\lim_{F \to G}\int{\frac{\varphi(F)}{|F|}}d\nu = \int{\phi}d\nu.
$$
\end{proof}

Given $\nu \in \mathcal{P}_G(X)$, the {\bf Kolmogorov-Sinai entropy} of $\nu$ is defined as
$$
h(\nu) = \lim_{F \to G} \frac{H_\nu(F)}{|F|},
$$
where $H_\nu(F)$ denotes the Shannon entropy of the partition $\{[w]: w \in X_F\}$. See \cite[\S 9]{kerr2016ergodic}.

\begin{theorem}[Variational principle]
\label{thm:varprin}
Let $X$ be a $G$-subshift. If $\varphi \in \Maps$ is asymptotically additive, then there exists $\phi \in \mathcal{C}(X)$ such that
$$
\vertG{\varphi - S(\phi)} = 0.
$$
If $X$ is compact, then
$$
\hat{P}(\varphi) = \sup\left\{h(\nu) + \lim_{F \to G} \int{\frac{\varphi(F)}{|F|}}d\nu: \nu \in \mathcal{P}_G(X)\right\},
$$
and if $X = \N^G$ and $\phi$ is exp-summable and uniformly continuous potential with finite oscillation, the same conclusion holds.
\end{theorem}

\begin{proof}
Due to Lemma \ref{lem:varprin1}, it follows that
$$
\sup\left\{h(\nu) + \lim_{F \to G} \int{\frac{\varphi(F)}{|F|}}d\nu: \nu \in \mathcal{P}_G(X)\right\} = \sup\left\{h(\nu) + \int{\phi}d\nu: \nu \in \mathcal{P}_G(X)\right\}.
$$

If $X$ is compact or $X = \N^G$ and $\phi$ is exp-summable and uniformly continuous potential with finite oscillation (see \cite{beltran2024thermodynamic}), then
$$
\sup\left\{h(\nu) + \int{\phi}d\nu: \nu \in \mathcal{P}_G(X)\right\} = P(\phi).
$$

Due to Proposition \ref{prop:varprin}, $P(\phi) = \hat{P}(\varphi)$, and we conclude.
\end{proof}

\begin{remark}
More generally, we can say that if $\vertG{\varphi - S(\phi)} = 0$, $P(\phi) = \hat{P}(\varphi)$, and the variational principle holds for $\phi$, then the variational principle holds for $\varphi$.
\end{remark}

The variational principle provides a criterion to choose relevant invariant measures among the, sometimes very large, set of invariant measures. Indeed, measures attaining the supremum in Theorem \ref{thm:varprin} are relevant since they capture all the weighted disorder seen by the pressure. We say that $\mu \in \mathcal{P}_G(X)$ is an {\bf equilibrium state} for $\varphi$ if
$$
\hat{P}(\varphi) = h(\mu) + \lim_{F \to G} \int{\frac{\varphi(F)}{|F|}}d\mu.
$$

Observe that, if $\vertG{\varphi-S(\phi)}=0$, then
$$
\hat{P}(\varphi) = P(\phi) = h(\mu) + \int{\phi}d\mu = h(\mu) + \lim_{F \to G} \int{\frac{\varphi(F)}{|F|}}d\mu,
$$
so the equilibrium states for $\varphi$ coincide with those for $\phi$.

In the compact case, the equilibrium states for $\phi$ are guaranteed to exist due to the upper semi-continuity of the entropy map $\nu \mapsto h(\nu)$ and the compactness of $\mathcal{P}_G(X)$.

\subsection{Some applications to the one-dimensional case}

\label{subsec:examples-thermo}
Let $k \in \N$ and consider the $\Z$-full shift $\Z \acts \{1, \dots ,k\}^\Z$ where $\Z$ acts by translations and let $V = \mathcal{C}( \{1, \dots ,k\}^\Z)$. The Koopman representation is given by $\kappa: \Z \to \Isom(V)$, where $\pi(-n)f(x) = f(n \cdot x)$ for every $n \in \Z$ and $x \in  \{1, \dots ,k\}^\Z$. Given $W \subseteq V$, we know that, by Remark \ref{rem:tool}, every $f \in \overline{W}$ is an additive realization of some set map $\varphi: \mathcal{F}(\Z) \to V$ with $\varphi(\Z) \subseteq W$.

We stress that the fine thermodynamic properties of $\varphi$ can be reduced to that of its additive realizations. Indeed, if $\varphi$ has a Lipschitz additive realization, then we have the following consequences (see \cite{bowen, parry1990zeta}):
\begin{enumerate}
    \item There exists a unique equilibrium state $\mu$ for $\varphi$.
    \item The measure $\mu$ is Gibbs, satisfies the Central Limit Theorem and it has exponential decay of correlations with respect to Lipschitz observables.
    \item The pressure function $t \mapsto P(t\varphi)$ is real analytic.
\end{enumerate}

In the following examples we will study different settings that show that an additive realization of certain asymptotically additive set maps  $\varphi: \FSet \to V$ may belong to $W$ or not. The existence of an additive realization in $W$ reduces to study the question whether $(f + \wCob) \cap W$ is empty or not, where $f$ is any additive realization of $\varphi$ in $V$.

\begin{example} \label{ex:hol}
Let $W \subseteq V$ be the subset of Lipschitz functions. It is well-known that $\overline{W} = V$.
\begin{itemize}

\item $\witness \cap W = \emptyset$: By \cite[Theorem V.2.2]{israel}, there exists $f \in V$ with uncountably many ergodic equilibrium states. By Remark \ref{rem:tool}, since $\left<W\right> = W$, there exists a set map $\varphi: \mathcal{F}(\Z) \to W$ that is asymptotically additive relative to $W$ with $f$ as an additive realization. On the other hand, every Lipschitz function has a unique equilibrium state. By Lemma \ref{lem:witness}, $\witness \cap W = (f + \wCob) \cap W = \emptyset$, because the set of equilibrium states for each function in $f + \wCob$ is the same. Therefore, $\varphi$ is asymptotically additive relative to $W$ but has no additive realization in $W$.

\item $\witness \cap W = \emptyset$: By \cite{hofbauer}, there exists $f \in V$ for which the pressure function $t \mapsto P(tf)$ is not real-analytic. By Remark \ref{rem:tool}, since $\left<W\right> = W$, there exists a set map $\varphi: \mathcal{F}(\Z) \to W$ that is asymptotically additive relative to $W$ with $f$ as an additive realization. On the other hand, for every Lipschitz function the pressure function is real-analytic \cite[Chapter 4]{parry1990zeta}. By Lemma \ref{lem:witness}, $\witness \cap W = (f + \wCob) \cap W = \emptyset$, because the pressure functions coincides for each function in $f + \wCob$. Therefore, $\varphi$ is asymptotically additive relative to $W$ but has no additive realization in $W$.

\item $\witness \cap W \neq \emptyset$ and $\witness \cap V \setminus W \neq \emptyset$: Let $f \in W$ and $h \in \wCob \setminus W$. By Remark \ref{rem:tool}, since $\left<W\right> = W$, there exists a set map $\varphi: \mathcal{F}(\Z) \to W$ that is asymptotically additive relative to $W$ with $f+h$ as an additive realization. Notice that $f+h$ is not in $W$, so $\witness \cap V \setminus W \neq \emptyset$. On the other hand, $\witness \cap W \neq \emptyset$, because $\witness = ((f+h)+\wCob)$, $-h \in \wCob$, and $f = (f+h)-h \in W$. Therefore, $\varphi$ is asymptotically additive relative to $W$, and has additive realizations in $W$ and its complement. Notice that Theorem \ref{thm:cuneo-rel} guarantees that $\witness \subseteq \overline{W + \Cob}$ and, since $\witness \cap W \neq \emptyset$, it also implies that $\witness \cap V \setminus (W + \wCob) = \emptyset$, which is compatible with the fact that $\witness \cap V \setminus W \neq \emptyset$ because $W + \wCob \subsetneq \overline{W + \Cob}$.
\end{itemize}
\end{example}

\begin{remark}
We emphasize that questions about the regularity of additive realizations are more commonly addressed in the setting of almost additive sequences rather than asymptotically additive ones. In \cite{bricenoiommi2}, we discuss a notion of almost additivity that shows how the results of this article on asymptotic additivity are also pertinent to the study of that framework. Furthermore, we proved that the developed theory permits the extension of several ergodic theorems to certain non-additive settings.
\end{remark}

\subsection*{Acknowledgments}

We thank the anonymous referees for their helpful suggestions, which improved an earlier version of this work.

\subsection*{Funding}

R.B. was partially supported by ANID/FONDECYT Regular 1240508. G.I. was partially supported by ANID/FONDECYT Regular 1230100.

\bibliographystyle{abbrv}
\bibliography{references}

\end{document}